\newif\ifMariaHasTheToken
\MariaHasTheTokenfalse                               % |
\newcommand{\MariaIsEditing}[1]{\ifMariaHasTheToken #1\fi}

\documentclass[12pt]{amsart}
\usepackage[utf8]{inputenc}
\usepackage{thm-restate}

\usepackage{enumerate,amsthm,todonotes, amssymb,xcolor,comment,subcaption}
\usepackage[footskip=1cm, headheight = 16pt, top=3.7cm, bottom=3cm,  right=2.5cm,  left=2.5cm ]{geometry}
\usepackage[colorlinks,linkcolor=blue,citecolor=red]{hyperref}
\usepackage{geometry}
\newtheorem{question}{Question}[section]
\newtheorem{lemma}[question]{Lemma}
\newtheorem{theorem}[question]{Theorem}
\newtheorem{conjecture}[question]{Conjecture}
\newtheorem{corollary}[question]{Corollary}

\newcounter{tbox}
\newcommand{\sta}[1]{\vspace*{0.3cm}\refstepcounter{tbox}\noindent{ \parbox{\textwidth}{(\thetbox) \emph{#1}}}\vspace*{0.3cm}}

\usepackage[T1]{fontenc}
\usepackage[leqno]{amsmath}
\usepackage{todonotes}
\usepackage[shortlabels]{enumitem}
\usepackage{latexsym}
\usepackage{amsfonts}
\usepackage{tikz}
\usepackage{float}
\usepackage{lmodern}
\usepackage[foot]{amsaddr}
\usepackage{latexsym}
\usepackage{tikzit}
\tikzstyle{box}=[shape=rectangle, text height=1.5ex, text depth=0.25ex, yshift=0.5mm, fill=white, draw=black, minimum height=5mm, yshift=-0.5mm, minimum width=5mm, font={\small}]
\tikzstyle{gate}=[shape=rectangle, text height=1.5ex, text depth=0.25ex, yshift=0.5mm, fill=white, draw=black, minimum height=5mm, yshift=-0.5mm, minimum width=5mm, font={\small}, tikzit category=circuit]
\tikzstyle{big gate}=[shape=rectangle, text height=1.5ex, text depth=0.25ex, yshift=0.5mm, fill=white, draw=black, minimum height=10mm, yshift=-0.5mm, minimum width=5mm, font={\small}, tikzit category=circuit]
\tikzstyle{Z dot}=[inner sep=0mm, minimum size=2mm, shape=circle, draw=black, fill={rgb,255: red,221; green,255; blue,221}, tikzit category=zx]
\tikzstyle{Z phase dot}=[minimum size=5mm, font={\footnotesize\boldmath}, shape=rectangle, rounded corners=2mm, inner sep=0.2mm, outer sep=-2mm, scale=0.8, tikzit shape=circle, draw=black, fill={rgb,255: red,221; green,255; blue,221}, tikzit draw=blue, tikzit category=zx]
\tikzstyle{X dot}=[Z dot, shape=circle, draw=black, fill={rgb,255: red,255; green,136; blue,136}, tikzit category=zx]
\tikzstyle{X phase dot}=[Z phase dot, tikzit shape=circle, tikzit draw=blue, fill={rgb,255: red,255; green,136; blue,136}, font={\footnotesize\boldmath}, tikzit category=zx]
\tikzstyle{hadamard}=[fill=yellow, draw=black, shape=rectangle, inner sep=0.6mm, minimum height=1.5mm, minimum width=1.5mm, tikzit category=zx]
\tikzstyle{paulibox}=[fill={rgb,255: red,221; green,221; blue,255}, draw=black, shape=rectangle, inner sep=0.6mm, minimum height=5mm, minimum width=5mm, font={\footnotesize}, text height=1.5ex, text depth=0.25ex, tikzit category=zx]
\tikzstyle{vertex}=[inner sep=0mm, minimum size=1mm, shape=circle, draw=black, fill=black, tikzit category=misc]
\tikzstyle{vertex set}=[inner sep=0mm, minimum size=2mm, shape=circle, draw=black, fill=white, font={\footnotesize\boldmath}, tikzit category=misc]
\tikzstyle{small black dot}=[fill=black, draw=black, shape=circle, inner sep=0pt, minimum width=1.2mm, tikzit category=circuit]
\tikzstyle{cnot ctrl}=[fill=black, draw=black, shape=circle, inner sep=0pt, minimum width=1.2mm, tikzit category=circuit]
\tikzstyle{cnot targ}=[fill=white, draw=white, shape=circle, tikzit category=circuit, label={center:$\oplus$}, inner sep=0pt, minimum width=2.1mm, tikzit fill={rgb,255: red,102; green,204; blue,255}, tikzit draw=black]
\tikzstyle{ket}=[fill=white, draw=black, shape=regular polygon, regular polygon sides=3, regular polygon rotate=-30, scale=0.7, inner sep=1pt, tikzit category=circuit, tikzit shape=rectangle, tikzit fill=green]
\tikzstyle{bra}=[fill=white, draw=black, shape=regular polygon, regular polygon sides=3, regular polygon rotate=30, scale=0.7, inner sep=1pt, tikzit category=circuit, tikzit shape=rectangle, tikzit fill=red]
\tikzstyle{scalar}=[shape=rectangle, text height=1.5ex, text depth=0.25ex, yshift=0.5mm, fill=white, draw=black, minimum height=5mm, yshift=-0.5mm, minimum width=5mm, font={\small}]
\tikzstyle{clabel}=[fill=white, draw=none, shape=rectangle, tikzit fill={rgb,255: red,56; green,255; blue,242}, font={\footnotesize}, inner sep=1pt, tikzit category=labels]
\tikzstyle{empty diagram}=[draw={gray!40!white}, dashed, shape=rectangle, minimum width=1cm, minimum height=1cm, tikzit category=misc]
\tikzstyle{white dot}=[Z dot]
\tikzstyle{gray dot}=[X dot]
\tikzstyle{white phase dot}=[Z phase dot]
\tikzstyle{gray phase dot}=[X phase dot]
\tikzstyle{small hadamard}=[hadamard]

\tikzstyle{simple}=[-]
\tikzstyle{hadamard edge}=[-, dashed, dash pattern=on 2pt off 0.5pt, thick, draw={rgb,255: red,68; green,136; blue,255}]
\tikzstyle{box edge}=[-, dashed, dash pattern=on 2pt off 0.5pt, thick, draw={rgb,255: red,203; green,192; blue,225}]
\tikzstyle{brace edge}=[-, tikzit draw=blue, decorate, decoration={brace,amplitude=1mm,raise=-1mm}]
\tikzstyle{diredge}=[->]
\tikzstyle{double edge}=[-, double, shorten <=-1mm, shorten >=-1mm, double distance=2pt]
\tikzstyle{gray edge}=[-, {gray!60!white}]
\tikzstyle{pointer edge}=[->, very thick, gray]
\tikzstyle{boldedge}=[-, line width=1.6pt, shorten <=-0.17mm, shorten >=-0.17mm]
\tikzstyle{implies}=[-implies, double, double distance=2pt]

\input{graph.tikzdefs}
\usepackage{algorithm}%
\usepackage{algorithmicx}%
\usepackage{algpseudocode}%
\usepackage{comment}
\usepackage{physics}
\usepackage{graphicx}
\usepackage{mathtools}
\usepackage{amssymb}
\usepackage{bm}
\usepackage[capitalise]{cleveref}
\usepackage{amsthm}
\usepackage{empheq}
\usepackage{amsmath}

\makeatletter
\newcommand{\leqnomode}{\tagsleft@true}
\newcommand{\reqnomode}{\tagsleft@false}
\makeatother

\def\dd{\hbox{-}}

\usepackage{bm}
\usepackage{bbm}
\usepackage{thmtools}

\newcommand{\E}[1]{\mathbb{E}\left[#1\right]}

\newcommand{\proba}[1]{\mathbb{P}\left(#1\right)}

\newcommand{\set}[1]{\left\{#1\right\}}
\newcommand{\sm}{\setminus}
\newcommand{\mt}{\emptyset}

\newcommand{\nat}{{\mathbb N}}

\newcommand{\NN}{\mathbb{N}}

\theoremstyle{definition}

\tikzset{every picture/.style={line width=0.75pt}} %set default line width to 0.75pt

\title{Induced minors and subpolynomial treewidth}

\author{Maria Chudnovsky$^{\dagger\mathsection}$}
\author{Julien Codsi$^{\ast \mathsection}$}
\author{David Fischer$^{\mathsection}$}
\author{Daniel Lokshtanov$^{\ddagger}$}

\thanks{$^{\ddagger}$ Department of Computer Science, University of California Santa Barbara, Santa Barbara, CA, USA. Supported by NSF Grant  CCF-2505099.}
\thanks{$^{\dagger}$ Supported by NSF Grants DMS-2348219 and   CCF-2505100, 
AFOSR grant FA9550-25-1-0275, and a Guggenheim Fellowship.}
\thanks{$^{\mathsection}$Princeton University, Princeton, NJ, USA}
\thanks{$^{\ast}$ Supported by NSF Grant DMS-2348219 and by the Fonds de recherche du Québec through the doctoral research scholarship 321124}

\begin{document}

\begin{abstract}
Given a family $\mathcal{H}$ of graphs, we say that a graph $G$ is $\mathcal{H}$-induced-minor-free if no induced minor of $G$ is isomorphic to a member of $\mathcal{H}$,
We denote by $W_{t\times t}$  the $t$-by-$t$ hexagonal grid, and by $K_{t,t}$ the complete bipartite graph with both sides of the bipartition of size $t$.  We show that the class of 
$\{K_{t,t},W_{t\times t}\}$-induced minor-free graphs with bounded clique number has subpolynomial  treewidth.
Specifically, we prove that for every integer $t$ there exist $\epsilon \in (0,1]$ 
and $c \in \mathbb{N}$ such that every $n$-vertex $\{K_{t,t},W_{t\times t}\}$-induced minor-free graph with no clique of size $t$  has treewidth at most $2^{c\log^{1-\epsilon}n}$. 
\end{abstract}

\maketitle
\MariaIsEditing{\textcolor{red}{DO NOT MODIFY! Maria has the token!}}

\section{Introduction}

All graphs in this paper are finite and simple, and all logarithms are base $2$.
For standard graph theory terminology that is not defined here we refer to reader to 
\cite{diestel}.
Let $G = (V(G),E(G))$ be a graph. For a set $X \subseteq V(G),$ we denote by $G[X]$ the subgraph of $G$ induced by $X$, and by $G \setminus X$ the subgraph of $G$ induced by $V(G) \setminus X$. In this paper, we use induced subgraphs and their vertex sets interchangeably. For subsets $X,Y \subseteq V(G)$ we say that $X$ is {\em complete} to $Y$ if $X$ and $Y$ are disjoint and every vertex of $X$ is adjacent to every vertex of $Y$, and that $X$ is {\em anticomplete} to $Y$
if $X$ and $Y$ are disjoint and every vertex of $X$ is non-adjacent to every vertex of $Y$.

For graphs $G$ and $H$, we say that $H$ is an {\em induced minor} of $G$ if
there exist disjoint connected induced subgraphs $\{X_v\}_{v \in V(H)}$ of $G$ 
such that $X_u$ is anticomplete to $X_v$ if and only if $u$ is non-adjacent to $v$ in $H$; in this case we say that $G$ {\em contains an $H$-induced-minor}.
Given a family $\mathcal{H}$ of graphs, we say that a graph $G$ is {\em $\mathcal{H}$-induced-minor-free} if no induced minor of $G$ is isomorphic to a member of $\mathcal{H}$.

For a graph $G$, a \emph{tree decomposition} $(T, \chi)$ of $G$ consists of a tree $T$ and a map $\chi\colon V(T) \to 2^{V(G)}$ with the following properties:
\begin{enumerate}
\itemsep -.2em
    \item For every $v \in V(G)$, there exists $t \in V(T)$ such that $v \in \chi(t)$.

    \item For every $v_1v_2 \in E(G)$, there exists $t \in V(T)$ such that $v_1, v_2 \in \chi(t)$.

    \item For every $v \in V(G)$, the subgraph of $T$ induced by $\{t \in V(T) \mid v \in \chi(t)\}$ is connected.
\end{enumerate}

For each $t\in V(T)$, we refer to $\chi(t)$ as a \textit{bag of} $(T, \chi)$.  The \emph{width} of a tree decomposition $(T, \chi)$, denoted by $width(T, \chi)$, is $\max_{t \in V(T)} |\chi(t)|-1$. The \emph{treewidth} of $G$, denoted by $tw(G)$, is the minimum width of a tree decomposition of $G$.  Graphs of bounded treewidth  are well-understood both  structurally
\cite{RS-GMXVI} and algorithmically \cite{Bodlaender1988DynamicTreewidth}.

Let $f: \mathbb{N} \rightarrow \mathbb{N}$ be a function.
A class $\mathcal{C}$  of graphs has {\em treewidth bounded by $f$}  if 
every $n$-vertex graph in $\mathcal{C}$ has treewidth at most $f(n)$.
If $f$ can be taken to be a constant function, $\mathcal{C}$ is said to have {\em bounded treewidth}. The question of which classes defined by forbidden induced subgraphs or minors have bounded treewidth has received  significant  attention in recent years, including \cite{BHKM}, \cite{Hajebi}, \cite{Korhonen}, \cite{LOZIN2022103517}, a series of  papers involving some of the  authors of this manuscript, and others. However, a recent result of \cite{ABTV} suggests that this question is unlikely to have a nice  answer.  On the other hand, classes whose treewidth is bounded by a slow-growing function seem to be better behaved, and are still of interest from the algorithmic perspective.

A whole host of graph problems~\cite{cygan2015parameterized} (including {\sc Maximum Weight Independent Set}, {\sc Coloring}, {\sc Feedback Vertex Set}, {\sc Dominating Set}, etc.) admit algorithms with running time  $2^{{(tw(G))}^{O(1)}}n^{O(1)}$. These problems can be solved in quasi-polynomial time on graph classes whose treewidth is bounded by a poly-logarithmic function, and in sub-exponential time (in the $O(2^{n^\epsilon})$ for every $\epsilon > 0$ sense) on graph classes whose treewidth is bounded by  a subpolynomial function. Hence, whenever a graph class has treewidth upper-bounded by a poly-logarthmic (sub-polynomial) function, {\em none} of these problems can be NP-hard on this class unless every problem in \textsf{NP} can be solved in quasi-polynomial (or sub-exponential) time. 

%is considered highly unlikely.
%Thus \textsc{MWIS} is unlikely to be  \textsf{NP}-hard  on these  classes of graphs unless every problem in \textsf{NP} can be solved in quasi-polynomial (or sub-exponential) time, which is considered highly unlikely.

%A {\em clique} in a graph is a set of pairwise adjacent vertices, and 
%a  {\em stable (or independent) set} is a set of pairwise non-adjacent vertices. 
%
%Given a graph $G$ with weights on its vertices, the \textsc{Maximum Weight Independent Set (MWIS)} problem is the problem of finding a stable set in $G$ of maximum total weight.
%We will discuss \textsc{MWIS} here, but much of what we say applies to a wide variety of 
%algorithmic questions, as is explained in \cite{tw15}.
%\textsc{MWIS} is known to be \textsf{NP}-hard \cite{alphahard}, but it can be solved in time $2^{{(tw(G))}^{O(1)}}n^{O(1)}$. Thus \textsc{MWIS} can be solved in quasi-polynomial time on graph classes whose treewidth is bounded by a poly-logarithmic function, and in sub-exponential time (in the $O(2^{n^\epsilon})$ for every $\epsilon > 0$ sense) on graph classes whose treewidth is bounded by  a subpolynomial function. 
%Thus \textsc{MWIS} is unlikely to be  \textsf{NP}-hard  on these  classes of graphs unless every problem in \textsf{NP} can be solved in quasi-polynomial (or sub-exponential) time, which is considered highly unlikely.

A {\em clique} in a graph is a set of pairwise adjacent vertices, and 
a  {\em stable (or independent) set} is a set of pairwise non-adjacent vertices.
We denote by $W_{t\times t}$  the $t$-by-$t$ hexagonal grid, and by $K_{t,t}$ the complete bipartite graph with both sides of the bipartition of size $t$. 
For a positive integer $t$, we  denote by    $\mathcal{C}_t$  the class of $\{K_{t,t}, W_{t \times t}\}$-induced-minor-free graphs, and by $\mathcal{C}_t^*$ be the subclass of $\mathcal{C}_t$ consisting of all graphs with no clique of size $t$.
The following conjecture has become known in the area:
\begin{conjecture} [from \cite{TIV}]
    \label{conj:smalltreealph}
    For every $t \in \nat$, there is an integer $d=d(t)$ such  that every $n$-vertex graph  $G \in \mathcal{C}_t^*$ satisfies 
    $tw(G) \leq \log^d n$.
\end{conjecture}

Here, we prove a weakening of this, replacing the poly-logarithmic bound on treewidth 
by a subpolynomial one:
\begin{restatable}{theorem}{main}\label{thm:twbound}
  For every  $t \in \mathbb{N}$, 
  there exist $\epsilon=\epsilon(t) \in (0,1]$ and  $c=c(t)\in \nat$ 
  such that every $n$-vertex 
  graph $G$ in $\mathcal{C}_t^*$ satisfies
%  \red{
%$tw(G) \leq   2^{(\log \log n)^d\log^{1-\epsilon} n }$}
$tw(G) \leq   2^{c\log^{1-\epsilon} n }$.
\end{restatable}

We remark that in view of Lemma 3.6 of \cite{aboulker} and the main theorem of \cite{TW16},
\cref{thm:twbound} can be restated in the language of forbidden induced subgraphs instead of induced minors, but we will not do it here.

\cref{thm:twbound} has the following corollary:
\begin{corollary}\label{subpoly}
For every $t \in \mathbb{N}$ and for every $\epsilon>0$ there exists
$N \in \mathbb{N}$ such that every $n$-vertex graph $G \in \mathcal{C}_t^*$ with $n>N$ satisfies
$tw(G) < n^{\epsilon}$.
\end{corollary}

We make two remarks about Corollary~\ref{subpoly}. First, it implies that for every $t \in \mathbb{N}$ and 
every $\epsilon>0$, 
every problem which is solvable in time $2^{{(tw(G))}^{O(1)}}n^{O(1)}$ admits an algorithm with running time $O(2^{n^\epsilon})$ on graphs in $\mathcal{C}_t^*$.

Second, for every $G \in \{W_{n \times n}, K_{n,n}, K_n\}$ we have that $tw(G) \geq |V(G)|^{\frac{1}{2}}$. Thus, Corollary~\ref{subpoly} is best possible in the following sense. For every proper induced minor closed class ${\mathcal F}$ and integer $r$, let $\mathcal{F}_r^*$ be the subclass of $\mathcal{F}$ consisting of all graphs with no clique of size $r$.
%let ${\mathcal F}_r^*$ be the set of graphs in ${\cal F}$ that do not contain a $K_r$.
%If, for every $n>0$ ${\cal F}$ contains $K_{n,n}$ or $W_{n \times n}$, then  ${\cal F}_3^*$ contains 
%
%Since for $G \in \{W_{n \times n}, K_{n,n}, K_n\}$, $tw(G) \geq |V(G)|^{\frac{1}{2}}$, 
%every induced-minor-closed class ${\mathcal F}$ of graphs such that ${\mathcal F}_3^*$ satisfies the conclusion of \cref{subpoly} is contained in $\mathcal{C}$.
%for some $s \in \mathbb{N}$. 
%Moreover, 
Then, for every induced-minor-closed graph class ${\mathcal F}$ and every $r \geq 3$, either ${\mathcal F} \subseteq {\mathcal C}_t$ for some $t$, in which case $\mathcal{F}_r^*$ has subpolynomial treewidth, or for every $N \in \nat$ there is a graph $G \in \mathcal{F}_3^*$ with $|V(G)| > N$ such that $tw(G) \geq |V(G)|^{\frac{1}{2}}$.

\subsection{Definitions and notation.} 
We continue with a few more definitions that will be used throughout the paper.
Let $G$ be a graph. We denote by $cc(G)$ the set of connected components of $G$. For a vertex $v \in V(G)$ we denote by $N(v)$ the set of neighbors of $v$, and $N[v]$ denotes $N(v) \cup \set{v}$. We denote the set of vertices in $G$ at distance exactly $2$ from $v$ by $N^2(v)$. For a set $X \subseteq V(G)$ we denote by $N(X)$ the set of
all vertices of $G \setminus X$ that have a neighbor in $X$, and we let $N[X] = N(X) \cup X$.
A {\em path} in $G$  is an induced subgraph that is a path. The
{\em length} of a path  is the number of edges in it.
We denote by  $P=p_1 \dd \dots \dd p_k$
a path in $G$ where $p_ip_j \in E(G)$ if and only if $|j-i|=1$. We say that $p_1$ and $p_k$ are
the {\em ends} of $P$. The {\em interior} of $P$, denoted by
$P^*$, is the set $P \setminus \{p_1,p_k\}$.
For $i,j \in \{1, \dots. k\}$ we denote by $p_i \dd P \dd p_j$ the subpath of
$P$ with ends $p_i,p_j$.

Let $G$ be a graph and let $A,B\subseteq G$ be disjoint. We say that a set $X\subseteq V(G)\setminus(A\cup B)$ \textit{separates} $A$ from $B$ if for every connected component $D$ of $G\sm X$, $D\cap A =\mt$ or $D\cap B = \mt$.
Let $a,b \in V(G)$ be non-adjacent.
A set $X\subseteq V(G)\setminus\set{a,b}$ \textit{separates} $a$ from $b$ if for every connected component $D$ of $G\sm X$, $|D\cap\set{a,b}|\leq 1$. We also call $X$ an {\em  $a \dd b$-separator}.  We denote by $conn_G(a,b)$ the minimum size of an $a\dd b$-separator in $G$.

\subsection{Proof outline and organization}
First we prove \cref{bipartite lemma} that states that the number of edges in a bipartite graph in $\mathcal{C}_t$ is linear in the number of vertices on the smaller side of the bipartition, provided no two vertices on the other side are twins (this last assumption is necessary because of the example of a large star). The proof uses a result of \cite{MatijaPolyDegen} and some probabilistic arguments. This lemma is separate from the rest of the proof, and we believe it to be of independent interest.

Let us now describe the main proof. For this informal presentation we find it easier to go through the proof in reverse order. First we use a result of \cite{BHKM} that states that the edges of every graph in $\mathcal{C}_t^*$ can be partitioned into a small number of star forests. Then, following the ideas introduced in \cite{Korhonen} and developed in \cite{BHKM}, we reduce the problem of bounding the treewidth of  graphs in $\mathcal{C}_t^*$ to a subclass of consisting of what we call "$(F,r)$-based" graphs. A graph $G$ is {\em $(F,r)$-based} if  there exists  an induced star forest $F$ with no isolated vertices in $G$,  such that $G$ admits a tree decomposition in which every bag consists of at most $r$ objects, each of which is  a star of $F$  or a single vertex; moreover, this phenomenon persists in all induced subgraphs of $G$ (with $F$ modified appropriately).  $(F,r)$-based graphs are defined at the start of \cref{sec:based}.  This is the only place in the proof where we explicitly use the bound on the clique number; the rest of the proof only assumes that the graph at hand is $(F,r)$-based for some $F$ and $r$. This reduction is done in \cref{sec:useclique}.

The next observation is that graphs in $\mathcal{C}_t$ are {\em $(p,q)$-slim} for appropriately chosen parameters $p$ and $q$. This means that in every stable set of $p$ vertices there exists a pair $a,b$ such that there are no $q$ disjoint and pairwise anticomplete $a  \dd b$-paths in $G$ (we call such a pair {\em $q$-slim}). The main ingredient of this proof is a lemma, essentially proved in  \cite{TW17}, applied to an appropriate graph. Analyzing the outcomes of that lemma through the lens of the main theorem of \cite{TW16} gives the result. The details are explained in \cref{sec:slim}.

The next step is to reduce the task of bounding the treewidth of an $(F,r)$-based graph
in $\mathcal{C}_t$ to the question of separating $q$-slim pairs of vertices. This is immediate from a
theorem in \cite{tw7} if the clique number is bounded, but here we present a different proof, which is more in the spirit of \cite{TI2}, that does not use this assumption.  For precise definitions of the terms below, see \cref{sec:twbased}.
Let $G$ be an $(F,r)$-based graph in $\mathcal{C}_t$; it is enough to show that every normal weight function $w$ on $G$ admits a small $w$-balanced separator. Since $G$ is $(F,r)$-based, we can find (by repeatedly using the existence of the special tree decomposition in a sequence of induced subgraphs of $G$) pairwise anticomplete sets $Y_1, \dots, Y_p$, each of size
at most $r$, and such that $N[Y_i]$ is a balanced separator in $G \setminus \bigcup_{j < i}Y_j$.  Next, assume that for every $q$-slim pair of vertices $y_i \in Y_i$ and
$y_j \in Y_j$, there is a small $y_i \dd y_j$-separator $S_{y_iy_j}$ in $G$.
Let $C$ be the union of all such $S_{y_iy_j}$. Then $C$ is still small.
We may assume that some component $D$  of $G \setminus (C \cup \bigcup_{i=1}^pY_i)$ has
$w(D)> \frac{1}{2}$. From the choice of $Y_1, \dots, Y_p$, each $Y_i$ contains a vertex $v_i$ with a neighbor in $D$. Since $G$ is $(p,q)$-slim, some pair $(v_i,v_j)$ is $q$-slim.
But $S_{v_iv_j} \subseteq C$, and yet there is a $v_i \dd v_j$-path with interior in $D$, which is a contradiction. This proof is done in \cref{sec:twbased}.

Our next and final goal is to prove \cref{thm: sep slim general version}, asserting that every $q$-slim pair in an $(F,r)$-based  graph in $\mathcal{C}_t$ admits a small separator. This is the most novel part of the paper, where new ideas are needed.
Let $G$ be an $(F,r)$-based  graph in $\mathcal{C}_t$, and let $(a,b)$ be a $q$-slim pair of vertices in $G$, and assume that no small $a \dd b$-separator exists. Using the special tree decomposition and a lemma from \cite{tw15}, we can find a very large (but with size bounded as a function of $r$ and $t$) collection of 
$a \dd b$-separators $S_1, \dots, S_x$, all pairwise disjoint and anticomplete to each other, such that each $S_i$ has the following properties:
\begin{itemize}
\item $S_i$ has a partition $D_i \cup Y_i \cup X_i$.
\item $Y_i$ is stable.
\item $X_i \subseteq N(Y_i)$.
\item $|D_i \cup Y_i| \leq r$.
\end{itemize}
This is done in \cref{sec:based}. Let $D=\bigcup_{i=1}^xD_i$ and $M= \bigcup_{i=1}^x Y_i$. Then the sizes of $D$ and $M$ are still under control. Note that $M$ is a stable set. Next, we use \cref{bipartite lemma} and averaging arguments to produce an induced subgraph of $H$ of  $G \setminus D$ and a a subset $C$ of $M$  such that
\begin{itemize}
 \item  the size of $H$ is a small proportion of the size of $G$.
 \item the size of $C$ is bounded as a function of $t$ and $r$.
 \item every $a \dd b$-path in $G \setminus (D \cup M)$ contains a subpath with interior in $H$ 
 such that at least $p$ vertices of $C$ have neighbors in the subpath.
\end{itemize}
Such a triple $(H,C,D \cup M)$ is called a {\em good  $a \dd b$-barrier} and its existence is proved in \cref{sec:getbarrier}. 

Since $G$ is $(p,q)$-slim,  it follows from the definition of a good barrier that every 
$a \dd b$-path in $G \setminus (D \cup M)$ contains a subpath $Q$ such that
\begin{itemize}
    \item $Q \subseteq H$, and 
    \item there is a $q$-slim pair $(u,v)$ with $u,v \in C$ such that $Q$ contains a subpath with ends $u,v$.
\end{itemize}

Inductively (on the size of $H$) there is a small (subpolynomial in  $V(H)$) set $S \subseteq V(H)$ such that  every $q$-slim pair $(u,v)$ with  $u,v \in C$ 
is separated in $H \setminus S$.
But now  $S \cup M \cup D$ is an
$a \dd b$-separator in $G$. This argument is carried out in \cref{sec:sepbarred} and
it completes the proof.
See \cref{fig: schema} for a diagrammatic depiction of the outline of the proof of \cref{thm: sep slim general version}.
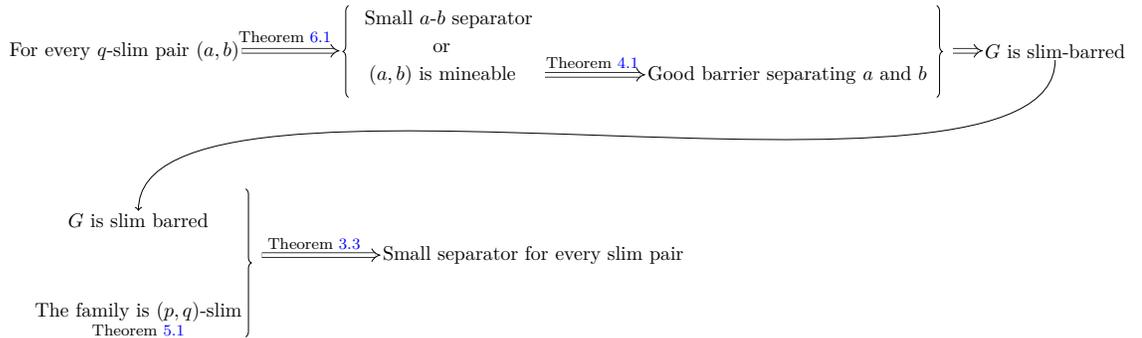
\begin{figure}[h]
    \centering
    \scalebox{0.75}{\begin{tikzpicture}
	\begin{pgfonlayer}{nodelayer}
		\node [style=clabel] (0) at (-3.5, 0) {For every $q$-slim pair $(a,b)$ };
		\node [style=clabel] (1) at (8, 1.1) {Small $a\dd b$ separator};
		\node [style=clabel] (2) at (7.75, 0.1) {or};
		\node [style=clabel] (3) at (8.25, -0.85) {$(a,b)$ is mineable   \ \ \ \ };
		\node [style=none] (5) at (4.25, 1.6) {};
		\node [style=none] (6) at (4.25, -1.65) {};
		\node [style=none] (7) at (4.125, 0) {};
		\node [style=clabel] (8) at (20, -0.85) {Good barrier separating $a$ and $b$};
		\node [style=clabel] (9) at (29.5, 0) {$G$ is slim-barred};
		\node [style=none] (10) at (25.5, 1.6) {};
		\node [style=none] (11) at (25.5, -1.65) {};
		\node [style=none] (12) at (25.875, 0) {};
		\node [style=clabel] (18) at (-3, -6) {$G$ is slim barred};
		\node [style=clabel] (19) at (-3, -9.25) {The family is $(p,q)$-slim};
		\node [style=none] (20) at (1, -4.9) {};
		\node [style=none] (21) at (1, -10.15) {};
		\node [style=none] (22) at (1.375, -7.25) {};
		\node [style=clabel] (23) at (11, -7.25) {Small separator for every slim pair};
		\node [style=none] (24) at (-3, -9.9) {\tiny  {\cref{thm:slim}}};
		\node [style=none] (25) at (13.1, -0.425) {\tiny {\cref{thm:mineabletobarrier}}};
		\node [style=none] (26) at (2, 0.5) {\tiny { \ \cref{thm:mineslim}}};
		\node [style=none] (27) at (3.225, -6.825) {\tiny {\cref{thm:slimbarredtosep}}};
	\end{pgfonlayer}
	\begin{pgfonlayer}{edgelayer}
		\draw [style=brace edge] (6.center) to (5.center);
		\draw [style=implies] (0) to (7.center);
		\draw [style=implies] (3) to (8);
		\draw [style=implies] (12.center) to (9);
		\draw [style=brace edge] (10.center) to (11.center);
		\draw [style=brace edge] (20.center) to (21.center);
		\draw [style=implies] (22.center) to (23);
		\draw [style=diredge, in=90, out=-90, looseness=0.50] (9) to (18);
	\end{pgfonlayer}
\end{tikzpicture}}
    \caption{Outline of the proof of the existence of small separators for $q$-slim pairs in $(F,r)$-based graphs in $\mathcal{C}_t$. } %\eqref{claim: not star}.}
    \label{fig: schema}
\end{figure}

Finally, we remark that the results in \cref{sec:sepbarred}  and \cref{sec:getbarrier} are stated in greater generality than what we need for the proof of \cref{thm:twbound}, as we expect them to be useful for other families of graphs that admit balanced separators with small domination number.

\section{The bipartite lemma} \label{sec:bipartite}

 Given a graph $G$, we say that $u,v\in V(G)$ are \textit{twins in $G$} if 
 $N(u)\setminus \{v\}=N(v)\setminus \{u\}$.
If $G$ is clear from the context, we will simply say that $u$ and $v$ are twins. A
{\em twin class} in $G$ is a maximal set of vertices, every two of which are twins. It is not hard to see that every graph has a unique partition into twin classes.

\begin{lemma} \label{bipartite lemma}
    There exists a function $f$ such that the following holds. Let $G$ be a bipartite $K_{t,t}$-induced-minor-free graph with bipartition $(A,B)$ such that no two vertices in $B$ are twins. Then $|E(G)| \leq f(t) |A|$.
\end{lemma}
\begin{proof}
    It follows from the main result of  \cite{MatijaPolyDegen} that there exists $\Delta = \Delta(t)$ such that $G$ is $\Delta$-degenerate. We may assume that $\Delta>2$. 
    \cite{MatijaPolyDegen} also implies the existence of $\Delta'=\Delta'(t)$ as a  degeneracy bound for graphs with no $K_{2t}$-subgraph and no induced subdivision of $K_{t,t}$.
     Let $f(t) = 2^{40\Delta^2\Delta'}\Delta+ (\Delta+1)\Delta $. Let us take a degenerate ordering of $V(G)$,
    that is an ordering $v_1, \dots, v_n$ in which each $v_i$ has at most  $\Delta$ neighbors in $\{v_{i+1}, \dots, v_n\}$.   We may assume that the ordering $v_1, \dots, v_n$
    was chosen so that  the vertices of $A$ appear as late as possible:  if $v_i \in A$ then all the vertices of degree at most $\Delta$ in $G\sm\set{v_1,\dots,v_{i-1}}$ are in $A$.

    \sta{If $1 \leq i < j \leq n$,  $v_i \in A$ and $\set{v_{i+1},\dots,v_j}\cap A = \mt$ then $j-i \leq \Delta$.\label{small_A_gaps}}
    
Let $i< l\leq j$. It follows from our choice of the ordering that 
$|N(v_l)\sm\set{v_1,\dots,v_{i-1}}|>\Delta$. Moreover, since $v_{i+1},\dots,v_{l-1}\in B$ and $B$ is stable, we have that
$\Delta \geq |N(v_l) \sm \set{v_1,\dots,v_{l-1}}|=|N(v_l) \sm \set{v_1,\dots,v_{i}}|.$
%\todo[inline]{$=\Delta$ above should be $>\Delta$, right?}
This implies that $v_l \in N(v_i)\sm\set{v_1,\dots,v_{i-1}}$. Since $ |N(v_i)\sm\set{v_1,\dots,v_{i-1}}| \leq \Delta$, \eqref{small_A_gaps} follows.
\\
\\

Let $s\in \nat$ be minimal such that $v_s\in A$. Let $B_1 = B \cap \set{v_s,\dots,v_n}$.
Let $G' = G\sm{B_1}$ and $B'= B\sm B_1$.  Note that every vertex in $B'$ has a degree at most $\Delta$ in $G'$.
 
 \sta{$|E(G')| \geq |E(G)|-|A|(\Delta+1)\Delta$. \label{E(G')}}

 Since $G$ is bipartite, $E(G)\sm E(G') = E(G[\set{v_s,\dots,v_n}])$. By \eqref{small_A_gaps}, $|\set{v_s,\dots,v_n}| \leq |A|(\Delta+1)$.
 Since $G[\set{v_s,\dots,v_n}]$ is $\Delta$-degenerate, $|E(G')| \geq |E(G)|- |A|(\Delta+1)\Delta$.
 This proves \eqref{E(G')}.
 \\
 \\

%Let $\Delta'$ be the degeneracy bound given by the main result of \cite{MatijaPolyDegen} for graphs with no $K_{2t}$-subgraph and no induced subdivision of $K_{t,t}$.
   % Let $r= 40\Delta^3$.
Let $r=40\Delta^2\Delta'$.
   Let $k\in \nat\cup\set{0}$ be maximal such that there exist $a_1,\dots,a_k \in A$ satisfying  that, for all $1 \leq i\leq k$, $|N^2_{G'\sm N_{G'}[ \set{a_1,\dots,a_{i-1}}]}(a_i)|\leq r$.
    Let $A_1 = \set{a_1,\dots,a_k }$ and let $A'' = A\sm A_1$ and $B'' = B'\sm N_{G'} (A_1)$. Let $G'' = G[A''\cup B'']$.

    \sta{$|E(G'')| \geq |E(G')| - 2^{r}\Delta|A_1|$ \label{claim remove small second neighborhoods}}

    Since there are no twins in $B$, we have that for every $i \leq k$, $|N_{G'}(a_i)\sm N[\set{a_1,\dots,a_{i-1}}]|\leq 2^{r}$. Thus, since 
    every vertex in $B'$ has a degree at most $\Delta$ in $G'$, 
    we deduce that $|E(G'')| \geq |E(G')| - 2^{r}\Delta|A_1|$,  proving \eqref{claim remove small second neighborhoods}.
    \\
    \\
    
    %Since $|N^2_{G'\sm N_{G'}[ \set{a_1,\dots,a_{i-1}}]}(a_i)|\leq r\Delta$, $G'$ is $K_{t,t}$-induced-minor-free, and every vertex $b\in B'$ has degree at least $t$ we have that $|N_{G'}(a_i)\sm N[\set{a_1,\dots,a_{i-1}}]|\leq \binom{r\Delta}{t}t \leq (r\Delta)^t$. Thus, $|E(G'')| \geq |E(G')| - (r\Delta)^t\Delta|A_1|$ proving \eqref{claim remove small second neighborhoods}.
    %\\
    %\\
    %\red{
    %If $|A''| \leq r \Delta^2 $, then $|E(G'')| \leq 2^{r\Delta^2}$ as there is no twins. %$$|E(G'')| \leq \binom{|A''|}{t} t \leq (r^2\Delta)^t t.$$ 
    %Combining this inequality with \eqref{E(G')} and \eqref{claim remove small second neighborhoods} implies that $$|E(G)| \leq \left(2^{r\Delta^2} + \Delta2^{r\Delta} +(\Delta+1)\Delta \right) |A|$$ and so we are done. %$|E(G)| \leq |A|\left((\Delta+1)\Delta)+(r\Delta)^2\right) +(r^2\Delta)^t t$ and so we are done.
    %Therefore, we might assume that $|A''|>r\Delta^2$.
    %}

    If $|A''|=0$ then \eqref{E(G')} and \eqref{claim remove small second neighborhoods} imply that $$|E(G)| \leq \left(2^{r}\Delta +(\Delta+1)\Delta \right) |A|$$ and so we are done. Therefore, we may assume for a contradiction that $|A''|>0$.

    \sta{Let $S = \set{\set{u,v} \text{ s.t \ } u,v\in A'' \text{ and } u\in N^2_{G''}(v)}$, then $|S| \geq \frac{r|A''|}{2}$ \label{handshake lemma number of potential good pairs}}

    Since for every $v\in A''$, we have $|N^2_{G''}(v)|> r$, there are at least $r|A''|$  ordered pairs of vertices in $A''$ at distance $2$. Dividing by $2$ accounts for the double counting. This proves \eqref{handshake lemma number of potential good pairs}.
    \\
    \\
    Sample $X\subseteq A''$ by iterating over every element of $A''$ and including it in $X$ with probability $p=\frac{1}{\Delta}$.
    %Let $X\subseteq A''$ be obtained by picking independently every element of $A''$ with probability $p=\frac{1}{\Delta}$.
    We will say that $u,v\in X$ is a \textit{good pair} if there exists $b\in B''$ for which $N(b)\cap X =\set{u,v}$. We now show that the expected number of good pairs is fairly high.
    %For $u\in A''$, let $v\in N^2_{G''}(u)$. 
    Let $\set{u,v}\in S$.
    We have that
    $$\proba{u,v\text{ is a good pair}} \geq \frac{1}{\Delta^2}\left(1-\frac{1}{\Delta}\right)^{\Delta-2}\geq \frac{1}{10 \Delta^2}, $$ as this bounds the probability of the event that $N(b)\cap X =\set{u,v}$ for $b\in B$ such that $\set{u,v}\subseteq N(b)$. 
    Therefore, by \eqref{handshake lemma number of potential good pairs}, $\E{\# \text{ good pairs}} = \sum_{\set{u,v}\in S} \proba{u,v\text{ is a good pair}} \geq  \frac{1}{10\Delta^2} \frac{r|A''|}{2}$.
    So there exists a choice of $X^*$ with at least $\frac{r|A''|}{20\Delta^2}$ good pairs. 
    Let $\Gamma$ be the graph with vertex set $X^*$, and where $u$ is adjacent to $v$ if and only if  $u,v$ is a good pair.

\sta{$E(\Gamma) \leq \Delta'|X^*|$.} \label{EGamma}
Since  $\Gamma$ is an induced minor of $G$, it follows that $\Gamma$ is  $K_{t,t}$-induced-minor-free, and in particular no induced subgraph of $\Gamma$ is a subdivision of $K_{t,t}$. Next suppose that there is a clique $K$ of size $2t$ in $\Gamma$;
let $K=\{k_1, \dots, k_{2t}\}$. It follows from the definition of a good pair that for
every $1 \leq j<j \leq 2t$ there exists $b_{ij} \in B$ such that $N(b_{ij}) \cap K=\{k_i,k_j\}$. But then $K \cup \{b_{ij}\}_{1 \leq i <j \leq 2t}$ is an induced 
subdivision of $K_{2t}$ in $G$, contrary to the fact that $G$ is $K_{t,t}$-induced-minor-free. This proves that $\Gamma$ has no clique of size ${2t}$. Now
the main result of \cite{MatijaPolyDegen} implies that $\Gamma$ is $\Delta'$-degenerate, and therefore $|E(\Gamma)|\leq \Delta' |X^*|$.
This proves~\eqref{EGamma}.
\\
\\
    On the other hand, by the choice of $X^*$ we  have that $|E(\Gamma)| \geq \frac{r|A''|}{20\Delta^2} \geq \frac{r|X^*|}{20\Delta^2} \geq 2\Delta' |X^*|$,
    contrary to \eqref{EGamma}. Hence $|A''|=0$, concluding the proof. 
\end{proof}

\section{Separating slim pairs in barred graphs} \label{sec:sepbarred}

Let $G=(V,E)$ be a graph, let $a,b\in V$ be non-adjacent and let $t,s\in\nat$. We say that the pair $(a,b)$ is  {\em $s$-wide} if there exist $s$ internally anticomplete
  $a \dd b$-paths in $G$; a pair of non-adjacent vertices that is not $s$-wide is said to be {\em $s$-slim}.
We say that $G$ is \textit{$(t,s)$-slim} if for every stable set $S\subseteq V$ of size $t$, there exist $a,b\in S$ such that $(a,b)$ is $s$-slim. Similarly, we say that a graph class $\mathcal{F}$ is \textit{$(t,s)$-slim} if every graph in $\mathcal{F}$ is $(t,s)$-slim.

In this section, we take the first step to our next goal: showing that an $s$-slim pair can be separated by a small subset of vertices. To do so we define the notion of a "barrier".
Loosely speaking, a barrier separating $a$ from $b$ is a relatively small induced subgraph $F$ of $G$ such that, in order to separate $a$ from $b$ in $G$, it suffices to delete a few vertices from $G \setminus F$ and then separate a few slim pairs from each other in $F$.
Now fix a slim pair $(a,b)$.
Assuming that such a barrier exists for {\em every} $s$-slim pair for an appropriately chosen $s$ (which is a property we called "slim-barred"), we design a recursive procedure to obtain a "small" $a\dd b$-separator in $G$. 
This reduces the problem of separating slim pairs of vertices to the problem of finding good barriers. In this section, we define and analyze this reduction.

Let $G=(V,E)$ be a graph, let $C,X,Y,Z$ be disjoint (and possibly empty) subsets of $V$ and let $t,p\in \nat$. Let $G'=G\sm C$.
We say that $B=(X,Y,Z,C)$ is a {\em $(t,p)$-barrier} if the following hold:
\begin{enumerate}
    \item\label{prop:barrier-interior-path} Every $X \dd Z$ path $P$ in $G'$ contains an $X \dd Z$ path $P'$ such that the interior of $P'$ is contained in $Y$. 
    \item\label{prop:barrier-see-t-comps} For every $X \dd Z$ path $P$ in $G'$, $| \set{\mathcal{C} \in cc(C) \mid N(P)\cap \mathcal{C} \neq\mt }| \geq t$.
    \item\label{prop:barrier-component-size-p} Every connected component of $C$ has at most $p$ vertices.
\end{enumerate}
We say that two $(t,p)$-barriers $B = (X,Y,Z,C)$ and $B' = (X',Y',Z',C')$ are \emph{disjoint} if $X \cup Y \cup Z$ is disjoint from $X' \cup Y' \cup Z'$. Similarly, we say that two $(t,p)$-barriers $B = (X,Y,Z,C)$ and $B' = (X',Y',Z',C')$ are \emph{anticomplete} if $X \cup Y \cup Z$ is anticomplete to $X' \cup Y' \cup Z'$. Let $u,v\in V\sm(X\cup Y\cup Z\cup C)$. We say that $B$ separates $u$ from $v$ (in $G$) if both $X\cup C$ and $Z\cup C$ separate $u$ from $v$ in $G$. See \cref{fig: barrier} for an illustration.
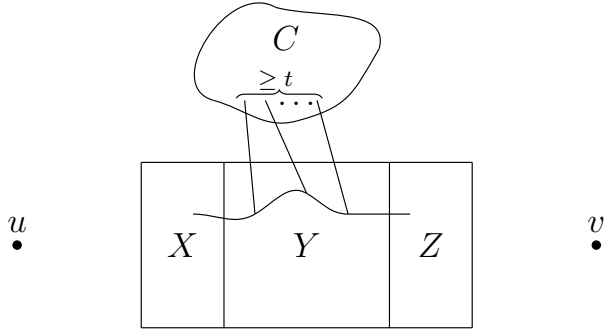
\begin{figure}[h]
    \centering
    \scalebox{1.1}{\begin{tikzpicture}
	\begin{pgfonlayer}{nodelayer}
		\node [style=none] (0) at (-2, 2) {};
		\node [style=none] (1) at (-4, 2) {};
		\node [style=none] (2) at (2, 2) {};
		\node [style=none] (3) at (4, 2) {};
		\node [style=none] (4) at (4, -2) {};
		\node [style=none] (5) at (2, -2) {};
		\node [style=none] (6) at (-2, -2) {};
		\node [style=none] (7) at (-4, -2) {};
		\node [style=none] (8) at (-2.25, 3.5) {};
		\node [style=none] (9) at (-0.75, 5.75) {};
		\node [style=none] (10) at (1.75, 4.75) {};
		\node [style=none] (11) at (-0.25, 3) {};
		\node [style=none] (12) at (-3, 0) {$X$};
		\node [style=none] (13) at (3, 0) {$Z$};
		\node [style=none] (14) at (0, 0) {$Y$};
		\node [style=vertex] (15) at (-7, 0) {};
		\node [style=vertex] (16) at (7, 0) {};
		\node [style=none] (17) at (-7, 0.5) {$u$};
		\node [style=none] (18) at (7, 0.5) {$v$};
		\node [style=none] (19) at (-2.75, 0.75) {};
		\node [style=none] (20) at (-1.25, 0.75) {};
		\node [style=none] (21) at (0, 1.25) {};
		\node [style=none] (22) at (1, 0.75) {};
		\node [style=none] (23) at (2.5, 0.75) {};
		\node [style=none] (24) at (-1.5, 3.5) {};
		\node [style=none] (25) at (-1, 3.5) {};
		\node [style=none] (26) at (-0.25, 3.425) {\small $\ldots$};
		\node [style=none] (27) at (0.25, 3.5) {};
		\node [style=none] (28) at (-1.7, 3.75) {};
		\node [style=none] (29) at (0.375, 3.75) {};
		\node [style=none] (30) at (-0.75, 4) {\tiny $\geq t$ };
		\node [style=none] (31) at (-0.5, 5) {$C$};
	\end{pgfonlayer}
	\begin{pgfonlayer}{edgelayer}
		\draw (7.center) to (1.center);
		\draw (1.center) to (0.center);
		\draw (0.center) to (2.center);
		\draw (2.center) to (3.center);
		\draw (3.center) to (4.center);
		\draw (4.center) to (5.center);
		\draw (5.center) to (6.center);
		\draw (6.center) to (7.center);
		\draw (6.center) to (0.center);
		\draw (2.center) to (5.center);
		\draw [in=150, out=165, looseness=1.25] (8.center) to (9.center);
		\draw [in=75, out=-15] (9.center) to (10.center);
		\draw [in=15, out=-120, looseness=1.25] (10.center) to (11.center);
		\draw [in=345, out=-165] (11.center) to (8.center);
		\draw [in=-150, out=0] (19.center) to (20.center);
		\draw [in=150, out=30] (20.center) to (21.center);
		\draw [in=-180, out=-30] (21.center) to (22.center);
		\draw (22.center) to (23.center);
		\draw (20.center) to (24.center);
		\draw (21.center) to (25.center);
		\draw (27.center) to (22.center);
		\draw [style=brace edge] (28.center) to (29.center);
	\end{pgfonlayer}
\end{tikzpicture}}
    \caption{Visualization of a barrier separating $u$ from $v$.} %\eqref{claim: not star}.}
    \label{fig: barrier}
\end{figure}

We say that a $(t,p)$-barrier $B = (X,Y,Z,C)$ is \textit{reduced} if  every connected component of $X\cup Y\cup Z$ intersects both $X$ and $Z$.

\begin{lemma}\label{lemma: reducing barriers}
    Let $B=(X,Y,Z,C)$ be a $(t,p)$-barrier separating $u$ from $v$. Then there exist $X'\subseteq X,Y'\subseteq Y, Z'\subseteq Z$ such that $(X',Y',Z',C)$ is a reduced $(t,p)$-barrier separating $u$ from $v$.
\end{lemma}

\begin{proof}
    Let $\Gamma$ be the union of all connected components in $X\cup Y\cup Z$ meeting both $X$ and $Z$.
    Let $X'=X\cap \Gamma, Y'=Y\cap \Gamma, Z'=Z\cap \Gamma$.
    
    \sta{$(X',Y',Z',C)$ is a $(t,p)$-barrier. \label{claim: reduced still barrier}}

    Condition \ref{prop:barrier-interior-path} holds as every $X' \dd Z'$ path in $G\sm C$ contains a path with interior in $Y$, and therefore in $Y'$. Condition \ref{prop:barrier-see-t-comps} holds as every $X' \dd Z'$ path is also a $X \dd Z$ path. Condition \ref{prop:barrier-component-size-p} holds since $C$ is unchanged.  This proves \eqref{claim: reduced still barrier}.
    \\
    \\
    \sta{$X'\cup C$ separates $u$ from $v$. \label{claim: reduced barier still separates}}
    
    Suppose not. Then there is a $u \dd v$-path $P$ with $P\cap (X'\cup C) = \mt$. Since $(X,Y,Z,C)$ separates $u$ from $v$, it follows that  $P \cap X \neq \emptyset$
    and $P \cap Z \neq \emptyset$, and consequently $P$
    contains an $X \dd Z$ path $Q$. Since $(X,Y,Z,C)$ is a barrier, we may assume that
    $Q^* \subseteq Y$. But then $Q \cap X' \neq \emptyset$,
    a contradiction. This proves \eqref{claim: reduced barier still separates}.
    \\
    \\
    Similarly, $Z'\cup C$ separates $u$ from $v$ and, thus,  $(X',Y',Z',C)$ separates $u$ from $v$.
\end{proof}

Let $p,s,t\in \nat$ and let $c,f,g: \nat\rightarrow \nat$ be functions such that $c(n)+f(n)+g(n)<n$. We say that an $n$-vertex graph $G$ is \textit{$(s,t,p,c,f,g)$-slim-barred} if for every $s$-slim pair $(u,v)$ there exist disjoint subsets $X,Y,Z,C,M$ of $V\sm \set{u,v}$ such that
\begin{enumerate}
    \item $B=(X,Y,Z,C)$ is a $(t,p)$-barrier in $G\sm M$ that separates $u$ from $v$.
    \item $|cc(C)| \leq c(n)$.
    \item $|M|\leq f(n)$.
    \item $|X\cup Y\cup Z| \leq g(n)$.
\end{enumerate}

Similarly, we say that a graph class $\mathcal{F}$ is \textit{$(s,t,p,c,f,g)$-slim-barred} if every graph in $\mathcal{F}$ is $(s,t,p,c,f,g)$-slim-barred.

\begin{lemma}
\label{lemma: slim barred to sep recursion}
    Let $s,t\in \nat$ and let $c,f,g: \nat\rightarrow \nat$ be functions such that $g(n)<n$.
    Let $\mathcal{H}$ be a hereditary graph class that is $(t,s)$-slim and $(s,t,p,c,f,g)$-slim-barred. Let 
    $$h(n)=\max_{G \in \mathcal{H}, \; |V(G)|\leq n} \max_{\; \; \; (a,b) \text{ is an  }s \text{-slim pair in } G} conn_G(a,b).$$
    Then, $h$ obeys the following recursive inequality.
 
    $$h(n) \leq     \begin{cases}
       n &\quad\text{if } n\leq 10\\
       f(n)+3(c(n)p)^2 + (c(n)p)^2h(g(n)) &\quad\text{otherwise} \\
     \end{cases}$$
\end{lemma}

\begin{proof}We proceed by induction on $n$. If $n\leq 10$, then the statement holds trivially. 
    Let $G$ be an $n$-vertex graph in $\mathcal{H}$ and $(a,b)$ be an $s$-slim pair in $G$.
    Let $X,Y,Z,C,M$ be disjoint subsets of $V\sm \set{a,b}$ as in the definition of $(s,t,p,c,f,g)$-slim-barred for $a$ and $b$. Let $I= \set{\set{u,v} \big| u,v \in C \text{ and } u,v \text{ is an $s$-slim pair in } X\cup Y\cup Z \cup \set{u,v}}$. 
    %Let $G'=X\cup Y\cup Z\cup C$ and let $M_2\subseteq G'\sm C$ be such that for every $s$-slim pair $(u',v')$ of $G'$ with
    %$u',v' \in C$, $M_2$ separates $u'$ from $v'$ in $G'$.
    For every pair $\set{u,v} \in I$, let $M_{u,v}$ be a $u \dd v$ separator in  $X\cup Y\cup Z \cup \set{u,v}$ of size $conn_{X \cup Y \cup Z \cup \{u,v\}}(u,v)$. Let $M'= \bigcup_{\set{u,v}\in I} M_{u,v}$.

    \sta{$M\cup M' \cup C$ separates $a$ from $b$ in $G$. \label{claim: recursion because separates 2 slim} }
    Suppose not and let $P$ be a path from $a$ to $b$ in $G\sm (M\cup M' \cup C)$.
    Since $(X,Y,Z,C)$ separates $a$ from $b$ in $G$, both $P\cap X$ and $P\cap Z$ are non-empty. Therefore, $P$ contains an $X \dd Z$ path. Since $(X,Y,Z,C)$ is a $(t,p)$-barrier, it follows that  
    %$P$ contains an $X \dd Z$ path $P'$ such that $P'^*\subseteq Y$ and  $N(P')$ 
    $N(P)$ meets at least $t$ connected components of $C$. Since $\mathcal{H}$ is $(t,s)$-slim, there exist $u,v  \in N(P)\cap C$ 
    such that the pair $(u,v)$ is $s$-slim. Then  $M_{u,v} \subseteq M'$.
    But since both $u$ and $v$ have neighbors in $P$,  there is a 
    $u \dd v$-path with interior in $P$. This is     a contradiction, 
    and \eqref{claim: recursion because separates 2 slim} follows.
\\
\\
\sta{For all $\set{u,v} \in I$, $|M_{u,v}| \leq h(g(n))+2$. \label{claim size min separator is kinda sublinear}}
If $|M_{u,v}|\leq 2$, the statement trivially holds, so we may assume that $|M_{u,v}|\geq 2$.
Let $c,d\in M_{u,v}$ and let $M_{u,v}'$ be a $u\dd v$ separator in $\left(X\cup Y\cup Z \cup \set{u,v}\right)\sm \set{c,d}$ with $|M_{u,v}'|$ minimum. Then, by induction, we have $|M_{u,v}|\leq 2 + |M_{u,v}'| \leq 2+ h(g(n))$. This proves \eqref{claim size min separator is kinda sublinear}.
\\
\\

    Since  $C$ contains at most $|C|^2\leq (c(n)p)^2$ slim pairs, and using both \eqref{claim: recursion because separates 2 slim} and \eqref{claim size min separator is kinda sublinear}, we get an $a \dd b$ separator of size  $$|M|+|C|+|M'|\leq f(n)+c(n)p+ (c(n)p)^2(h(g(n))+2)\leq f(n)+3(c(n)p)^2+ (c(n)p)^2h(g(n)).$$
\end{proof}

\begin{theorem}
\label{thm:slimbarredtosep}

    Let $s,t\in \nat$ and let $c,f,g: \nat\rightarrow \nat$ be increasing functions such that $g(n)<n$ and such that the ratio $\frac{n}{g(n)}$ is non-increasing.
    Let $\mathcal{H}$ be a hereditary graph class that is $(t,s)$-slim and $(s,t,p,c,f,g)$-slim-barred. Let 
    $$h(n)=\max_{G \in \mathcal{H}, \; |V(G)|\leq n} \max_{\; \; \; (a,b) \text{ is an  }s \text{-slim pair in } G} conn_G(a,b).$$
    Then, $h(n) \leq 20(f(n)+3c(n)^2p^2) (c(n)p)^{ \frac{2\log{(n)}}{\log\left(\frac{n}{g(n)}\right)}}$ 
\end{theorem}

Intuitively, \cref{thm:slimbarredtosep} follows by analyzing the recursion tree for the function $h$ obtained by \cref{lemma: slim barred to sep recursion}. It has a depth of at most $\frac{\log{(n)}}{\log\left(\frac{n}{g(n)}\right)}$ and each of its nodes has at most  $(c(n)p)^2$ children; consequently it has a most $2  (c(n)p)^{ \frac{2\log{(n)}}{\log\left(\frac{n}{g(n)}\right)}}$ vertices. The contribution of each vertex of the recursion tree is, at most, $10(f(n)+3c(n)^2p^2)$. We now proceed with a formal proof.

\begin{proof} [Proof of \cref{thm:slimbarredtosep}.]

Let $N \in \nat$. Let 
$$H(n) = \begin{cases}
       n &\quad\text{if } n\leq 10\\
       f(N)+3(c(N)p)^2 + (c(N)p)^2H(g(n)) &\quad\text{otherwise} \\
     \end{cases}.$$
By \cref{lemma: slim barred to sep recursion} and since $f$ and $c$ are increasing, we have that $h(n)\leq H(n)$ for all $n\leq N$, so it is sufficient to prove that $H(N) \leq 20(f(N)+3c(N)^2p^2) (c(N)p)^{ \frac{2\log{(N)}}{\log\left(\frac{N}{g(N)}\right)}}$.

We prove the following slightly stronger statement: for all $n\leq N$, we have
$$H(n) \leq 20(f(N)+3c(N)^2p^2) (c(N)p)^{ \frac{2\log{(n)}}{\log\left(\frac{n}{g(n)}\right)}} - \frac{(f(N)+3c(N)^2p^2)}{c(N)^2p^2 -1}.$$
Let $K= f(N)+3(c(N)p)^2$ and $z = c(N)^2p^2$.
We proceed by induction on $n$. If $n\leq 10$, then $20K - \frac{K}{z-1} \geq 19K \geq 19 \geq n $.

Therefore, we may assume that  $n >10$ and that for all $n'\leq n$, $H(n') \leq 20K\ z^{ \frac{\log{(n')}}{\log\left(\frac{n'}{g(n')}\right)}} - \frac{K}{z-1}$.

Now  we have that
%\todo[inline]{I am not sure why the equality from line 3 to line 4 below holds—I %think that maybe it is supposed to be $\leq$ instead of $=$, but even then, I do not %see why the inequality holds...}
\begin{align*}
    H(n) &\leq K + z\ H(g(n))\\
    &\leq K + z\left(20K\ z^{ \frac{\log{(g(n))}}{\log\left(\frac{g(n)}{g(g(n))}\right)}} - \frac{K}{z-1}\right)\\
    &\leq K + z\left(20K\ z^{ \frac{\log{(g(n))}}{\log\left(\frac{n}{g(n)}\right)}} - \frac{K}{z-1}\right)\\
    &=K + z\left(20K z^{ \frac{\log{\left(n \frac{g(n)}{n}\right)}}{\log\left(\frac{n}{g(n)}\right)}} - \frac{K}{z-1}\right)\\
    &=K + z\left(20K z^{ \frac{\log{(n)}}{\log\left(\frac{n}{g(n)}\right)}-1} - \frac{K}{z-1}\right)\\
    &=K + 20K\ z^{ \frac{\log{(n)}}{\log\left(\frac{n}{g(n)}\right)}} - \frac{zK}{z-1}\\
    &= 20K\ z^{ \frac{\log{(n)}}{\log\left(\frac{n}{g(n)}\right)}} - \left(\frac{z}{z-1}-1\right)K\\
    &= 20K\ z^{ \frac{\log{(n)}}{\log\left(\frac{n}{g(n)}\right)}} - \frac{K}{z-1}
\end{align*}
as required.

\end{proof}

\section{From mines to barriers} \label{sec:getbarrier}
In this section, we introduce the notion of ``mineable'' pairs of vertices. Informally, a pair $a,b$ is mineable if there exist many ``almost disjoint'' $a\dd b$-separators each with a small dominating set (a ``core''), and these cores are pairwise 
disjoint and anticomplete.
%such that each of them can be dominated by a ``small'' set of vertices (called {\em %cores}), and the cores of any two separators are disjoint and anti-complete.
We show that mineable pairs of vertices admit good barriers separating them (after 
deleting a small set of vertices from the graph). This is then used to prove that 
our graph class is slim-barred. We now proceed with formal definitions.
%to the problem of showing that every slim pair is either mineable or has a small separator.
%
In this paper the we will only apply mineability with $z=p=1$. We include the more general form for potential future applications. 

%\todo[inline]{tell reader that this is overkill and that we will use the full generality in a follow up. In this paper our graphs will be "(1,1,1,1) mineable}

Let $x,y,z, p\in \nat$ and let $G$ be a graph. For $a,b \in V(G)$,
we say that $(a,b)$ is $(x,y,z,p)$-mineable in $G$ if there exist disjoint  $Y_1,\dots,Y_x\subseteq V \setminus \{a,b\}$ for which the following hold:
% \begin{enumerate}
%     \item For every $i$, there exists $X_i' \subseteq N(Y_i) \setminus \{a,b\}$ such that    $Y_i \cup X_i$ is an $a \dd b$ separator  in $G \setminus \bigcup_{j<i} Y_j$. 
    
%     \item $\bigcup_{i \leq x} N(Y_i)$ is a stable set. 
%     \todo[inline]{Why do we have or need condition (2)?}
%     \item $\bigcup_{i\leq x} Y_i$ is a stable set.
%     \todo[inline]{For future, we need to replace  (3) with the assumption that $Y_i$ is anticomplete to $Y_j$ if $i \neq j$, and talk about contracting components of each $Y_i$ when tbyhe time comes. Otherwise I don't see how we can use this for general layered sets?}
%     \item For every $i$, $|Y_i|\leq y$.
%     \item Every vertex of $G$ is contained in at most $z$ of the sets $N[Y_1],\dots,N[Y_x]$.
% \end{enumerate}
\begin{enumerate}
    \item\label{prop:separators} For every $i$, there exists non-empty $X_i \subseteq N(Y_i) \setminus \{a,b\}$ such that $X_i' = Y_i \cup X_i$ is an $a \dd b$ separator  in $G \setminus \bigcup_{j<i} Y_j$. 
    \item\label{prop:absamecomp}  $a$ and $b$ belong the the same component of $G \setminus \bigcup_{j=1}^x Y_j$; and in particular 
    for every $i$, $a$ and $b$ belong to the same component of 
    $G \setminus \bigcup_{j<i} Y_j$. 
    \item\label{prop:separatorsanticomplete} For distinct $i,j \in [x]$, $Y_i$ and $Y_j$ are anticomplete.
    \item For every $i$, $|cc(Y_i)|\leq y$.
    \item For every $i$, every component of $Y_i$ has size at most $p$.
    \item Every vertex of $G$ is contained in at most $z$ of the sets $X_1',\ldots,X_x'$.
\end{enumerate}
The goal of this section is to show that every mineable pair (with appropriately chosen parameters) in a $K_{t,t}$-induced-minor-free graphs  can be separated by a barrier with certain properties:
\begin{theorem}
\label{thm:mineabletobarrier}
    There exists a function $\phi : \mathbb{N} \rightarrow \mathbb{N}$ such that the following holds.
    Let $t,p,x,y,z \in \nat$.
    Let $G$ be an $n$-vertex $K_{t,t}$-induced-minor-free graph and let $(a,b)$ be  an $(x,y,z,p)$-mineable pair in $G$. Then there exist disjoint subsets $X,Y,Z,C,M$ of $V(G)\sm \set{a,b}$ such that
\begin{enumerate}
    \item $B=(X,Y,Z,C)$ is a $(t,p)$-barrier in $G\sm M$ that separates $a$ from $b$.
    \item $|cc(C)|\leq \frac{100}{99} (4z+2t) \phi(t)yt$.
    \item $|M|\leq xyp$.
    \item $|X\cup Y\cup Z| \leq \frac{100n (4z+2t)^2}{x}$.
\end{enumerate}
\end{theorem}

% Let $F$ be an $a,v_{\alpha - 1}$-path in $G \sm C$ achieving $|N_G(F) \cap C| = d_a(v_{\alpha - 1}) < i_j^-$. Note that $v_\alpha \notin F$ as this would imply that $d_a(v_\alpha) \leq d_a(v_{\alpha - 1})$. Let $F'$ be an $a-v_\alpha$ path obtained by adding the edge $v_{\alpha-1}v_\alpha$ to the end of $F$. Since $d_a(v_\alpha) \geq i_j^-$, we have $|N_G(F') \cap C| \geq i_j^- > |N_G(F) \cap C|$. This implies that $N_G(v_\alpha) \cap C \neq \emptyset$, so $v_\alpha \in S_k$ for some $k$. We know that $k \geq i_j^-$ since $d_a(v_\alpha) \geq i_j^-$. We now observe that $|N_G(F') \cap C| \leq |N_G(F) \cap C| + |N_G(v_\alpha) \cap C| < i_j^- + |N_G(v_\alpha) \cap C|$. Since $(a,b)$ is $(x,y,z)$-mineable in $G$, we have $|N_G(v_\alpha) \cap C| \leq yz$.

We start with a lemma, which roughly states that the property of $(a,b)$ being $(x,y,z,p)$-mineable implies that we can find a sufficiently large number of small, pairwise anticomplete $(t,p)$-barriers in $G$ each of which separate $a$ from $b$.

\begin{lemma} \label{lemma: mines to barriers}
    Let $G$ be a graph with $n = |V(G)|$, let $x,y,z,t \in \nat$, and let $w = \lceil \frac{99}{100} \frac{x}{4z + 2t} \rceil$. Let $a,b \in V(G)$ and assume that $(a,b)$ is $(x,y,z,p)$-mineable in $G$. Then there exists a set $C \subseteq V(G) \setminus \set{a,b}$ with $|cc(C)| \leq xy$ and a collection $\mathcal{B} = \set{B_i = (I_i,J_i,K_i,C)}_{i \in [w]}$ of $(t,p)$-barriers, such that
    \begin{itemize}
        \item For every $i$, $(I_i,J_i,K_i,C)$ separates $a$ from $b$.
        \item For every $i$, $|I_i \cup J_i \cup K_i| \leq \frac{100n(4z + 2t)^2}{x}$.
        \item For all distinct $i,j \in [w]$, $B_i$ and $B_j$ are anticomplete.
        \end{itemize}
\end{lemma}
See \cref{fig: many barriers} for an illustration of the outcome of \cref{lemma: mines to barriers}.
\begin{figure}[h]
    \centering
    \scalebox{1}{\begin{tikzpicture}
	\begin{pgfonlayer}{nodelayer}
		\node [style=none] (0) at (-8, -0.5) {};
		\node [style=none] (1) at (-9.5, -0.5) {};
		\node [style=none] (2) at (-6, -0.5) {};
		\node [style=none] (3) at (-4.5, -0.5) {};
		\node [style=none] (4) at (-4.5, -3.5) {};
		\node [style=none] (5) at (-6, -3.5) {};
		\node [style=none] (6) at (-8, -3.5) {};
		\node [style=none] (7) at (-9.5, -3.5) {};
		\node [style=none] (8) at (-1.75, 2) {};
		\node [style=none] (9) at (-0.25, 4.25) {};
		\node [style=none] (10) at (2.25, 3.25) {};
		\node [style=none] (11) at (0.25, 1.5) {};
		\node [style=none] (12) at (-8.75, -2) {$I_1$};
		\node [style=none] (13) at (-5.25, -2) {$K_1$};
		\node [style=none] (14) at (-7, -2) {$J_1$};
		\node [style=vertex] (15) at (-11.75, -2) {};
		\node [style=vertex] (16) at (11.25, -2) {};
		\node [style=none] (17) at (-11.75, -1.5) {$a$};
		\node [style=none] (18) at (11.25, -1.5) {$b$};
		\node [style=none] (19) at (-8.5, -1.25) {};
		\node [style=none] (23) at (-5.75, -1.25) {};
		\node [style=none] (31) at (0, 3.5) {$C$};
		\node [style=none] (32) at (-1.75, -0.5) {};
		\node [style=none] (33) at (-3.25, -0.5) {};
		\node [style=none] (34) at (0.25, -0.5) {};
		\node [style=none] (35) at (1.75, -0.5) {};
		\node [style=none] (36) at (1.75, -3.5) {};
		\node [style=none] (37) at (0.25, -3.5) {};
		\node [style=none] (38) at (-1.75, -3.5) {};
		\node [style=none] (39) at (-3.25, -3.5) {};
		\node [style=none] (40) at (-2.5, -2) {$I_2$};
		\node [style=none] (41) at (1, -2) {$K_2$};
		\node [style=none] (42) at (-0.75, -2) {$J_2$};
		\node [style=none] (43) at (-2.25, -1.25) {};
		\node [style=none] (44) at (0.5, -1.25) {};
		\node [style=none] (58) at (6.25, -0.5) {};
		\node [style=none] (59) at (4.75, -0.5) {};
		\node [style=none] (60) at (8.25, -0.5) {};
		\node [style=none] (61) at (9.75, -0.5) {};
		\node [style=none] (62) at (9.75, -3.5) {};
		\node [style=none] (63) at (8.25, -3.5) {};
		\node [style=none] (64) at (6.25, -3.5) {};
		\node [style=none] (65) at (4.75, -3.5) {};
		\node [style=none] (66) at (5.5, -2) {$I_w$};
		\node [style=none] (67) at (9, -2) {$K_w$};
		\node [style=none] (68) at (7.25, -2) {$J_w$};
		\node [style=none] (69) at (5.75, -1.25) {};
		\node [style=none] (70) at (8.5, -1.25) {};
		\node [style=none] (71) at (3.2, -2) {$\dots$};
		\node [style=none] (72) at (-7.75, -1) {};
		\node [style=none] (73) at (-7, -1.25) {};
		\node [style=none] (74) at (-6.25, -1) {};
		\node [style=none] (75) at (-1.25, -1) {};
		\node [style=none] (76) at (-0.5, -1.25) {};
		\node [style=none] (77) at (0, -1) {};
		\node [style=none] (78) at (6.75, -1.25) {};
		\node [style=none] (79) at (7.375, -1.3) {};
		\node [style=none] (80) at (7.75, -1) {};
		\node [style=none] (81) at (-1.5, 2.75) {};
		\node [style=none] (82) at (-1, 2.25) {};
		\node [style=none] (83) at (-0.25, 2.45) {};
		\node [style=none] (84) at (0, 2.25) {};
		\node [style=none] (85) at (0.5, 2.25) {};
		\node [style=none] (86) at (0.75, 2.5) {};
		\node [style=none] (87) at (1.25, 2.5) {};
		\node [style=none] (88) at (-1, 2.8) {};
		\node [style=none] (89) at (-0.5, 2.15) {};
	\end{pgfonlayer}
	\begin{pgfonlayer}{edgelayer}
		\draw (7.center) to (1.center);
		\draw (1.center) to (0.center);
		\draw (0.center) to (2.center);
		\draw (2.center) to (3.center);
		\draw (3.center) to (4.center);
		\draw (4.center) to (5.center);
		\draw (5.center) to (6.center);
		\draw (6.center) to (7.center);
		\draw (6.center) to (0.center);
		\draw (2.center) to (5.center);
		\draw [in=150, out=165, looseness=1.25] (8.center) to (9.center);
		\draw [in=75, out=-15] (9.center) to (10.center);
		\draw [in=15, out=-120, looseness=1.25] (10.center) to (11.center);
		\draw [in=345, out=-165] (11.center) to (8.center);
		\draw (39.center) to (33.center);
		\draw (33.center) to (32.center);
		\draw (32.center) to (34.center);
		\draw (34.center) to (35.center);
		\draw (35.center) to (36.center);
		\draw (36.center) to (37.center);
		\draw (37.center) to (38.center);
		\draw (38.center) to (39.center);
		\draw (38.center) to (32.center);
		\draw (34.center) to (37.center);
		\draw (65.center) to (59.center);
		\draw (59.center) to (58.center);
		\draw (58.center) to (60.center);
		\draw (60.center) to (61.center);
		\draw (61.center) to (62.center);
		\draw (62.center) to (63.center);
		\draw (63.center) to (64.center);
		\draw (64.center) to (65.center);
		\draw (64.center) to (58.center);
		\draw (60.center) to (63.center);
		\draw (72.center) to (81.center);
		\draw (73.center) to (88.center);
		\draw (74.center) to (82.center);
		\draw (75.center) to (89.center);
		\draw (76.center) to (83.center);
		\draw (84.center) to (77.center);
		\draw (78.center) to (85.center);
		\draw (86.center) to (79.center);
		\draw (80.center) to (87.center);
	\end{pgfonlayer}
\end{tikzpicture}}
    \caption{Visualization of the result of \cref{lemma: mines to barriers}. }
    \label{fig: many barriers}
\end{figure}
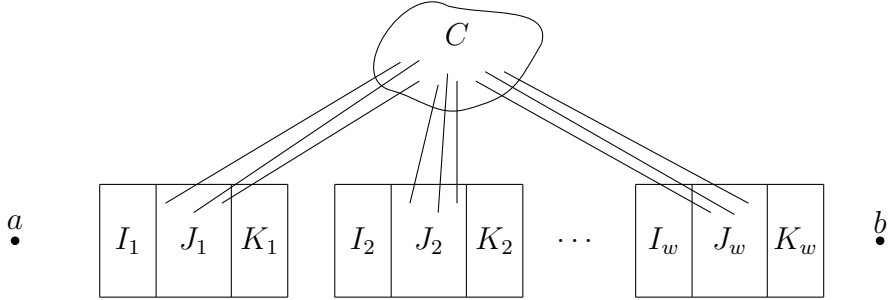

The proof of Lemma~\ref{lemma: mines to barriers} proceeds by a ``distance layering'' argument. 
We obtain a set $C$ and $a$-$b$ separators $X_1, \ldots, X_x$ in $G \sm C$ using that $(a,b)$ is mineable.
We define the ``distance from $a$'' for every vertex $v$ to be the number of separators $X_j$ we need to pass through in order to get from $a$ to $v$ in $G \sm C$. This acts much like a distance function, with an approximate triangle inequality and the property that walking along any edge changes the distance by at most $z$. This lets us define ``distance layers'' as all vertices that have a certain distance from $a$. Our barriers will be unions of not too many consecutive distance layers. The bound on the size of the layers follows from an averaging argument.

\begin{proof}
    Let $Y_1,\dots,Y_x$ and $X_1,\ldots,X_x$ be sets as in the definition of $(x,y,z,p)$-mineable, with $X_i' = Y_i \cup X_i$ for each $i$. Let $C = \bigcup_{i=1}^x Y_i$; observe that $|cc(C)| \leq xy$ as $|cc(Y_i)| \leq y$ for each $i$. For each subset $W \subseteq V(G) \sm C$, let $\psi(W) = \lvert \set{i \colon W \cap X_i \neq \emptyset}\rvert$. For every vertex $v\in G\sm C$ we define
\[ d_a(v):= \min_{a \dd v\text{-path $P$ in } G\sm C} \psi(P), \]
where $d_a(v) = \infty$ if there is no $a \dd v$ path in $G \sm C$. For $j \in \NN$ we denote by $S_j$ the set of vertices $v \in \bigcup_{i=1}^x X_i$ such that $d_a(v) = j$. 

% don't think we need any of this
% By property~(\ref{prop:absamecomp}) of mineability, $a$ and $b$ are in the same component of $G \sm C$. 
% Furthermore $X_1$ separates $a$ from $b$ in $G \setminus C$, and so $X_1$ contains a vertex $c$ in the same component of $G \sm C$ as $a$ and $b$. We have that $\ell' = d_a(c)$ is finite, so $c \in S_{\ell'}$.
% %
% Let $\ell \in \NN$ be the largest integer such that $S_\ell$ is nonempty; in particular, $\ell$ is finite. 

The following properties regarding $\psi$ and $d_a$ are immediate from the definitions, and are used implicitly in the analysis that follows.
\begin{itemize}
    \item \label{lem:dist-props:triangle} For $W,W' \subseteq V(G) \sm C$ we have $\psi(W \cup W') \leq \psi(W) + \psi(W')$. In particular, for distinct $v_1,v_2 \in V(G)$ and a path $P$ with ends $v_1$ and $v_2$, we have $d_a(v_2) \leq d_a(v_1) + \psi(P)$.
    \item \label{lem:dist-props:walk} If $W$ is a walk from $a$ to $v$ in $G \sm C$, then $d_a(v) \leq \psi(W)$.
    \item \label{lem:dist-props:S} If $v_1,v_2 \in V(G) \sm C$ are adjacent and $d_a(v_1) < d_a(v_2)$, then 
    $v_{2} \in S_k$ for some $k \in [d_a(v_1),d_a(v_1) + z]$.
\end{itemize}

    We now define $m = \left\lfloor \frac{x}{2z + t} \right\rfloor$, and for $j \in [m]$ we let $i_j^- = (j-1)(2z + t) + 1$ and $i_j^+ = i_j^- + z$. Write $W_j = \bigcup_{k=i_j^-}^{i_j^+} S_k$.  Then the sets  $W_j$ are pairwise disjoint. We let 
\[ L_{j} = \set{v \in G \sm \left(C \cup W_j \cup W_{j+1}\right) \colon i_j^- \leq d_a(v) < i_{j+1}^-}. \]

Next we prove the following.

\sta{\label{claim:Bj is a barrier} For $j \in [m - 1]$, $B_j = (W_j,L_j,W_{j+1},C)$ is a $(t,p)$-barrier.}

Recall that, for each $j$, no component of $Y_j$ has more than $p$ vertices. As $C$ is the union of the sets $Y_j$, and the sets  $Y_j$ are pairwise anticomplete, it follows that $B_j$ satisfies condition (\ref{prop:barrier-component-size-p}) in the definition of a $(t,p)$-barrier.

Next, we show that $B_j$ satisfies condition (\ref{prop:barrier-interior-path}) in the definition of a $(t,p)$-barrier. Let $P$ be a $W_j \dd W_{j+1}$ path in $G \setminus C$; we show that there is a subpath $P'$ of $P$ that is a $W_j \dd  W_{j+1}$ path with $(P')^* \in L_j$. If $P$ consists of a single edge then the statement holds trivially, so we may assume this is not the case. Write $P = v_1 \dd v_2 \dd \ldots \dd v_r$. If $d_a(v_i) \geq i_j^-$ for every $1 \leq i \leq r$, let $\alpha = 1$, otherwise let $\alpha = \max \set{i \colon 1 \leq i \leq r, \; d_a(v_i) < i_j^-} + 1$. Note that $\alpha \leq n$ as $d_a(v_n) \geq i_{j+1}^- > i_j^-$ since $P$ is a $W_j \dd W_{j+1}$ path.

We show that $v_\alpha \in W_j$. If $\alpha = 1$ then this is true because $P$ is a $W_j \dd  W_{j+1}$ path. Otherwise, by the choice of $\alpha$, we have $d_a(v_{\alpha - 1}) < d_a(v_{\alpha})$, so $v_\alpha \in S_k$ for some $k \in [d_a(v_{\alpha - 1}), d_a(v_{\alpha - 1}) + z]$. Furthermore, by the choice of $\alpha$ we have $v_\alpha \geq i_j^-$, and since $d_a(v_{\alpha - 1}) < i_j^-$ we have $d_a(v_{\alpha - 1}) + z < i_j^+$. Thus $k \in [i_j^-,i_j^+]$, so $v_\alpha \in W_j$.

We further observe that $\alpha < r$, since $d_a(v_\alpha) \leq i_j^+ < i_{j+1}^- \leq d_a(v_r)$. We may thus define $\beta = \min \set{i \colon \alpha < i \leq r, d(v_i) \geq i_{j+1}^-}$. An argument analogous to the previous paragraph shows that $v_\beta \in W_{j+1}$. 

It follows immediately from the definitions of $\alpha$ and $\beta$ that 
    $i_j^- \leq d_a(v_i) < i_{j+1}^-$
    for each $i$ such that $\alpha < i < \beta$. 
    In particular, we have that 
    $v_{\alpha} \dd P \dd v_{\beta}$
    is a $W_j \dd  W_{j+1}$ path contained in $W_j \cup L_j \cup W_{j+1}$.
    It follows that $v_{\alpha} \dd P \dd v_\beta$ contains a subpath (possibly itself) that is a $W_j \dd W_{j+1}$ path with interior in $L_j$. Thus, $B_j$ satisfies condition (\ref{prop:barrier-interior-path}) in the definition of a $(t,p)$-barrier.

    It remains to check condition (\ref{prop:barrier-see-t-comps}) in the definition of a $(t,p)$-barrier.
    In view of the first condition, 
    it is enough to show that every  $W_j \dd  W_{j+1}$ path $P$ with $P^* \subseteq L_j$ satisfies the second condition. Observe that if $\psi(P) \geq t$, then the second condition is satisfied as the sets $Y_i$ are pairwise anticomplete. 
    
    We now show that  $\psi(P) \geq t$.
    Let $u \in W_j$ and $v \in W_{j+1}$ be the ends of $P$, and let $P_1$ be a $a \dd u$-path in $G \sm C$ achieving $\psi(P_1) = d_a(u) \leq i_j^+$. Then appending $P$ to $P_1$ yields an $a \dd v$ walk $W$ in $G \sm C$. We thus have have
    \[ d_a(v) \leq \psi(W) \leq d_a(u) + \psi(P) \leq i_j^+ + \psi(P). \]
    On the other hand, since $v \in W_{j+1}$ we know that $d_a(v) \geq i_{j+1}^- \geq i_j^+ + t$. Consequently,  $i_j^+ + t \leq d_a(v) \leq i_j^+ + \psi(P)$, thus $\psi(P) \geq t$, and condition (\ref{prop:barrier-see-t-comps}) in the definition of a $(t,p)$-barrier is satisfied. This proves (\ref{claim:Bj is a barrier}).

\sta{\label{claim:Bj separates a and b} For $j \in [m-1]$, $B_j$ separates $a$ and $b$.}

It suffices to show that $W_\ell \cup C$ separates $a$ from $b$ for every $\ell \in [m]$. First,  since each $X_k'$ separates $a$ from $b$ in $G \sm C$, it follows that every $a \dd b$ path in $G \sm C$ meets $X_k'$ for all $k \in [x]$,
and therefore $d_a(b) \geq x$.

Now let $\ell \in [m]$, and let $P = v_1 \dd \ldots \dd v_r$ be an $a \dd b$ path in $G \sm C$, where $v_1 = a$ and $v_r = b$.
Observe that by the definition of $i_\ell^-$ it holds that $i_\ell^- < x$. Since $d_a(b) \geq x$, we can define $\alpha \in [r]$ to be maximal such that $d_a(v_\alpha) < i_\ell^-$; note that $\alpha \leq r-1$. We now have that $v_{\alpha + 1} \in S_k$ for some $k \in [d_a(v_\alpha),d_a(v_\alpha) + z]$. By the choice of $\alpha$ we have $k \geq i_j^-$, and since $d_a(v_\alpha) < i_\ell^-$ we have $d_a(v_\alpha) + z < i_\ell^- + z = i_\ell^+$, so $v_{\alpha + 1} \in W_\ell$. Since every $a \dd b$ path $P$ in $G \sm C$ meets  $W_\ell$, it follows that $W_\ell \cup C$ separates $a$ from $b$ in $G$, proving (\ref{claim:Bj separates a and b}).
\\
\\
Finally, we show that a large enough subset of the $B_j$ are pairwise anticomplete and have sufficiently small size. First, we put $m' = \lfloor \frac{m}{2} \rfloor = \lfloor \frac{x}{4z + 2t} \rfloor$ and $\mathcal{B}' = \set{B_{2i} \colon i \in [m']}$.

\sta{\label{claim: barriers anticomplete} For distinct $\alpha, \beta \in [m']$, $B_{2\alpha}$ and $B_{2\beta}$ are anticomplete.}

Assume without loss of generality that $\alpha < \beta$. We have $d_a(v) \leq i_{2\alpha + 1}^+$ for all $v \in (W_{2\alpha} \cup I_{2\alpha} \cup W_{2\alpha + 1})$ and $d_a(v) \geq i_{2\beta}^-$ for all $v \in (W_{2\beta} \cup I_{2\beta} \cup W_{2\beta + 1})$. Since $\alpha < \beta$ we have $2\beta - (2\alpha + 1) > 0$, and thus
    \[ i_{2\beta}^- - i_{2\alpha + 1}^+ = (2\beta - (2\alpha + 1))(2z+t) - z \geq 2z + t - z > 0. \]
    Since $d_a$ takes values at most $i_{2\alpha + 1}$ on vertices in $(W_{2\alpha} \cup L_{2\alpha} \cup W_{2\alpha + 1})$ and values at least $i_{2\beta}^-$ on vertices in $(W_{2\beta} \cup L_{2\beta} \cup W_{2\beta + 1})$, and $i_{2\alpha + 1}^+ < i_{2\beta}^-$, we conclude that $B_{2\alpha}$ and $B_{2\beta}$ are disjoint. Furthermore, suppose that $v_\alpha \in (W_{2\alpha} \cup L_{2\alpha} \cup W_{2\alpha + 1})$ and $v_\beta \in (W_{2\beta} \cup L_{2\beta} \cup W_{2\beta + 1})$ are adjacent. Then we have $v_\beta \in S_k$ for some $k \in [d_a(v_\alpha),d_a(v_\alpha) + z]$,  and consequently $d_a(v_\beta) \in [d_a(v_\alpha),d_a(v_\alpha) + z]$. But $d_a(v_\alpha) + z < i_{2\beta}^- \leq d_a(v_\beta)$, a contradiction. It follows that $B_{2\alpha}$ and $B_{2\beta}$ are anticomplete, proving (\ref{claim: barriers anticomplete}).
    \\
    \\
    We now define 
    \[ \mathcal{B}'' = \set{B_{2i} \colon i \in [m'], \; |W_{2i} \cup L_{2i} \cup W_{2i + 1}| > \frac{100n(4z + 2t)^2}{x}} \subseteq \mathcal{B}' \]
    and write $\mathcal{B} = \mathcal{B}' \sm \mathcal{B}''$, so that for each $B_{2i} \in \mathcal{B}$ it holds that $|W_{2i} \cup L_{2i} \cup W_{2i + 1}| \leq \frac{100n(4z + 2t)^2}{x}$.

    \sta{\label{claim:enough small barriers} It holds that $|\mathcal{B}| \geq \frac{99}{100} \frac{100n(4z + 2t)^2}{x}$.}

    Since the $B_{2i}$ (for $i \in [m']$) are pairwise disjoint, we have $\sum_{i = 1}^{m'} |W_{2i} \cup I_{2i} \cup W_{2i + 1}| \leq n$. Thus, we also have $\sum_{B_{2i} \in \mathcal{B}''} |W_{2i} \cup L_{2i} \cup W_{2i + 1}| \leq n$. Since 
    \[ \sum_{B_{2i} \in \mathcal{B}''} |W_{2i} \cup L_{2i} \cup W_{2i + 1}| > |\mathcal{B}''|\left(\frac{100n (4z + 2t)^2}{x} \right), \]
    we have $|\mathcal{B}''| < \frac{x}{100(4z+2t)^2} \leq \frac{x}{100(4z+2t)}$. Letting $\mathcal{B} = \mathcal{B}' \setminus \mathcal{B}''$, we have
    \[ |\mathcal{B}| = |\mathcal{B}'| - |\mathcal{B}''| > \left\lfloor \frac{x}{4z + 2t} \right\rfloor - \frac{x}{100(4z + 2t)}, \]
    and thus $|\mathcal{B}| \geq \frac{99}{100} \frac{x}{4z + 2t}$, proving (\ref{claim:enough small barriers}).
\\
\\
    We may now arbitrarily remove elements from $\mathcal{B}$ so that it has size $\lceil \frac{99}{100} \frac{x}{4z + 2t} \rceil$, completing the proof.
\end{proof}

Next we show that in the collection of barriers produced by \cref{lemma: mines to barriers} we can choose one with vertex set anticomplete to almost all components of $C$.
\begin{lemma} \label{thm: from many barrier one is good}
    There exists a function $\phi : \mathbb{N} \rightarrow \mathbb{N}$ such that the following holds.
    Let $\beta,p,t\in \nat$. Let $G$ be an $n$-vertex $K_{t,t}$-induced-minor-free graph and let $a,b\in V(G)$. Let $\mathcal{B} = \set{B_i = (X_i,Y_i,Z_i,C)}_{i=1}^\beta$ be a family of pairwise anticomplete $(t,p)$-barriers separating $a$ from $b$.
    Then there exists $C'\subseteq C$ with $|cc(C')|\leq \frac{\phi(t)|cc(C)|t}{\beta}$, $i^*$ and $X_{i^*}'\subseteq X_{i^*},Y_{i^*}\subseteq Y_{i^*}',Z_{i^*}'\subseteq Z_{i^*},$ such that $(X_{i^*}',Y_{i^*}',Z_{i^*}',C')$ is a $(t,p)$-barrier in $G\sm M$ where $M=C\sm C'$.
\end{lemma}
See \cref{fig: one good barrier out of many} for an illustration of the outcome of \cref{thm: from many barrier one is good}.
\begin{figure}[h]
    \centering
    \scalebox{1}{\begin{tikzpicture}
	\begin{pgfonlayer}{nodelayer}
		\node [style=none] (0) at (-9.5, -0.5) {};
		\node [style=none] (1) at (-11, -0.5) {};
		\node [style=none] (2) at (-7.5, -0.5) {};
		\node [style=none] (3) at (-6, -0.5) {};
		\node [style=none] (4) at (-6, -3.5) {};
		\node [style=none] (5) at (-7.5, -3.5) {};
		\node [style=none] (6) at (-9.5, -3.5) {};
		\node [style=none] (7) at (-11, -3.5) {};
		\node [style=none] (8) at (-3.25, 1.75) {};
		\node [style=none] (9) at (-1.75, 4.25) {};
		\node [style=none] (10) at (0.75, 3.25) {};
		\node [style=none] (11) at (-1.25, 1.25) {};
		\node [style=none] (12) at (-10.25, -2) {$X_1$};
		\node [style=none] (13) at (-6.75, -2) {$Z_1$};
		\node [style=none] (14) at (-8.5, -2) {$Y_1$};
		\node [style=vertex] (15) at (-12, -2) {};
		\node [style=vertex] (16) at (9, -2) {};
		\node [style=none] (17) at (-12, -1.5) {$a$};
		\node [style=none] (18) at (9, -1.5) {$b$};
		\node [style=none] (19) at (-10, -1.25) {};
		\node [style=none] (20) at (-7.25, -1.25) {};
		\node [style=none] (21) at (-1.5, 3.5) {$C$};
		\node [style=none] (22) at (-2.55, -0.5) {};
		\node [style=none] (23) at (-4.05, -0.5) {};
		\node [style=none] (24) at (-0.55, -0.5) {};
		\node [style=none] (25) at (0.95, -0.5) {};
		\node [style=none] (26) at (0.95, -3.5) {};
		\node [style=none] (27) at (-0.55, -3.5) {};
		\node [style=none] (28) at (-2.55, -3.5) {};
		\node [style=none] (29) at (-4.05, -3.5) {};
		\node [style=none] (30) at (-3.3, -1.75) {$X'_{i^*}$};
		\node [style=none] (31) at (0.2, -1.75) {$Z'_{i^*}$};
		\node [style=none] (32) at (-1.55, -1.75) {$Y'_{i^*}$};
		\node [style=none] (33) at (-3.05, -1.25) {};
		\node [style=none] (34) at (-0.3, -1.25) {};
		\node [style=none] (35) at (4.5, -0.5) {};
		\node [style=none] (36) at (3, -0.5) {};
		\node [style=none] (37) at (6.5, -0.5) {};
		\node [style=none] (38) at (8, -0.5) {};
		\node [style=none] (39) at (8, -3.5) {};
		\node [style=none] (40) at (6.5, -3.5) {};
		\node [style=none] (41) at (4.5, -3.5) {};
		\node [style=none] (42) at (3, -3.5) {};
		\node [style=none] (43) at (3.75, -2) {$X_\beta$};
		\node [style=none] (44) at (7.25, -2) {$Z_\beta$};
		\node [style=none] (45) at (5.5, -2) {$Y_\beta$};
		\node [style=none] (46) at (4, -1.25) {};
		\node [style=none] (47) at (6.75, -1.25) {};
		\node [style=none] (48) at (1.95, -2) {$\dots$};
		\node [style=none] (49) at (-9.25, -1) {};
		\node [style=none] (50) at (-8.5, -1.25) {};
		\node [style=none] (51) at (-7.75, -1) {};
		\node [style=none] (52) at (-2.05, -1) {};
		\node [style=none] (53) at (-1.3, -1.25) {};
		\node [style=none] (54) at (-0.8, -1) {};
		\node [style=none] (55) at (5, -1.25) {};
		\node [style=none] (56) at (5.625, -1.3) {};
		\node [style=none] (57) at (6, -1) {};
		\node [style=none] (58) at (-3.25, 3) {};
		\node [style=none] (59) at (-2.75, 2.5) {};
		\node [style=none] (60) at (-1.75, 1.85) {};
		\node [style=none] (61) at (-1.375, 1.875) {};
		\node [style=none] (62) at (-1, 2.5) {};
		\node [style=none] (63) at (-0.75, 2.75) {};
		\node [style=none] (64) at (-0.25, 2.75) {};
		\node [style=none] (65) at (-2.75, 3.05) {};
		\node [style=none] (66) at (-2, 1.65) {};
		\node [style=none] (67) at (-5.05, -2) {$\dots$};
		\node [style=none] (68) at (-2.5, 2.175) {};
		\node [style=none] (69) at (-1.75, 2.675) {};
		\node [style=none] (70) at (-1, 1.925) {};
		\node [style=none] (71) at (-2.25, 1.5) {};
		\node [style=none] (72) at (-1.75, 2.15) {$C'$};
		\node [style=none] (73) at (-4.05, -2.75) {};
		\node [style=none] (74) at (0.95, -2.75) {};
	\end{pgfonlayer}
	\begin{pgfonlayer}{edgelayer}
		\draw (7.center) to (1.center);
		\draw (1.center) to (0.center);
		\draw (0.center) to (2.center);
		\draw (2.center) to (3.center);
		\draw (3.center) to (4.center);
		\draw (4.center) to (5.center);
		\draw (5.center) to (6.center);
		\draw (6.center) to (7.center);
		\draw (6.center) to (0.center);
		\draw (2.center) to (5.center);
		\draw [in=150, out=165, looseness=1.25] (8.center) to (9.center);
		\draw [in=75, out=-15] (9.center) to (10.center);
		\draw [in=15, out=-120, looseness=1.25] (10.center) to (11.center);
		\draw [in=345, out=-165] (11.center) to (8.center);
		\draw (29.center) to (23.center);
		\draw (23.center) to (22.center);
		\draw (22.center) to (24.center);
		\draw (24.center) to (25.center);
		\draw (25.center) to (26.center);
		\draw (26.center) to (27.center);
		\draw (27.center) to (28.center);
		\draw (28.center) to (29.center);
		\draw (28.center) to (22.center);
		\draw (24.center) to (27.center);
		\draw (42.center) to (36.center);
		\draw (36.center) to (35.center);
		\draw (35.center) to (37.center);
		\draw (37.center) to (38.center);
		\draw (38.center) to (39.center);
		\draw (39.center) to (40.center);
		\draw (40.center) to (41.center);
		\draw (41.center) to (42.center);
		\draw (41.center) to (35.center);
		\draw (37.center) to (40.center);
		\draw (49.center) to (58.center);
		\draw (50.center) to (65.center);
		\draw (51.center) to (59.center);
		\draw (52.center) to (66.center);
		\draw (53.center) to (60.center);
		\draw (61.center) to (54.center);
		\draw (55.center) to (62.center);
		\draw (63.center) to (56.center);
		\draw (57.center) to (64.center);
		\draw [in=-90, out=165, looseness=0.75] (71.center) to (68.center);
		\draw [in=-165, out=90, looseness=1.25] (68.center) to (69.center);
		\draw [in=45, out=15] (69.center) to (70.center);
		\draw [in=-30, out=-120, looseness=0.75] (70.center) to (71.center);
		\draw (73.center) to (74.center);
	\end{pgfonlayer}
\end{tikzpicture}}
    \caption{Visualization of \cref{thm: from many barrier one is good}. }
    \label{fig: one good barrier out of many}
\end{figure}

\begin{proof} 
    Let $\phi$ be  the function $f$  from \cref{bipartite lemma}. Let $\mathcal{B}' = \set{B_i' = (X_i',Y_i',Z_i',C)}_{i=1}^\beta$ be the family of reduced pairwise anticomplete $(t,p)$-barriers obtained by applying \cref{lemma: reducing barriers} on each member of $\mathcal{B}$.

    Let $\mathcal{D} =cc\left( \bigcup_{i\leq \beta} (X_i'\cup Y_i'\cup Z_i')\right)$ and $\mathcal{C}=cc(C)$. Now consider the bipartite graph $\Gamma$ with bipartition $(\mathcal{D},\mathcal{C})$ and where there is an edge from $d\in \mathcal{D}$ to $c\in \mathcal{C}$ if there is an edge from $d$ to $c$ in $G$. $\Gamma$ is an induced minor of $G$ as it can be obtained by deleting every vertex not in $C\cup \bigcup_{D\in \mathcal{D}} D$ and contracting every component of $\mathcal{D}$ and $\mathcal{C}$. Therefore, $\Gamma$ is $K_{t,t}$-induced-minor-free.  Let $\Gamma'$ be the induced subgraph of $\Gamma$  containing exactly one representative for each twin class of vertices in $\mathcal{D}$.

    \sta{$|E(\Gamma')| \geq \frac{|E(\Gamma)|}{t} $ \label{claim: small twin class}}

    The second condition of the definition of a $(t,p)$-barrier together with the fact that all barriers in $\mathcal{B}'$ are reduced imply that
    $\deg_\Gamma(d)\geq t$ for every $d  \in \mathcal{D}$. It follows that 
    every twin class in $\mathcal{D}$ contains fewer than $t$ elements, as otherwise, there is a $K_{t,t}$-induced-minor in $G$.
    %Moreover, contracting the connected component of $C$ can only decrease the number of edges by a factor of at most $p$. 
    This proves \eqref{claim: small twin class}.
\\
\\
    The graph $\Gamma'$ satisfies the assumptions of  \cref{bipartite lemma}, and therefore  $|E(\Gamma')| \leq \phi(t,k) |\mathcal{C}|$.
    By \eqref{claim: small twin class}, $|E(\Gamma)| \leq \phi(t,k) |\mathcal{C}|t$. 
    Therefore, there is  a $(t,p)$-barrier $B_{i^*}'$ and a set $\mathcal{C'} \subseteq \mathcal{C}$ such that $|\mathcal{C}'|\leq \frac{\phi(t) |\mathcal{C}|t}{\beta}$ and $N(X_{i^*}'\cup Y_{i^*}'\cup Z_{i^*}') \subseteq \bigcup_{c \in \mathcal{C'}}c$. Let $C' =  \bigcup_{c \in \mathcal{C'}}c$. Noting that each  component of $C'$ is  a component of $C$ and thus contains at most $p$ vertices, we deduce that  
     $(X_{i^*}',Y_{i^*}',Z_{i^*}',C')$ is a $(t,p)$-barrier in $G\sm M$ where $M=C\sm C'$ as required.
\end{proof}

Now we summarize what we have shown so far to prove the main result of this section:
    \begin{proof}[Proof of \cref{thm:mineabletobarrier}.]
        Let $\phi$ be defined as in \cref{thm: from many barrier one is good}.  
        The assertion of \cref{thm:mineabletobarrier}
        follows by combining the family of $(t,p)$-barriers obtained in \cref{lemma: mines to barriers} with \cref{thm: from many barrier one is good}.
    \end{proof}

\section{The class $\mathcal{C}_t$ is slim}\label{sec:slim}

The goal of this section is to prove the following:
\begin{theorem} \label{thm:slim}
  For every $t \in \nat$ there exist $p,q \in \nat$ such that
  the class  $\mathcal{C}_{t}$ is $(p,q)$-slim.
  \end{theorem}
We start with some definitions from \cite{TW17}.
Let $G,H$ be graphs. Let $V(H)=\{v_1, \dots, v_k\}$.
An {\em induced $H$-model in $G$} is a $k$-tuple 
  $K=(C_1,\ldots, C_k)$ of pairwise disjoint connected induced subgraphs of $G$ such that for all distinct $i,j \in \{1, \dots, k\}$, the sets $C_i$ and $C_j$ are anticomplete if an only if $v_i$ is non-adjacent to $v_j$ in $H$.
We say that $K$ is \emph{linear} if every $C_i$ is a path in $G$.

Let $G$ be a graph.
For $k,l \in \nat$, a {\em $(k,l)$-block} in $G$ is a pair $(B, \mathcal{P})$ where $B\subseteq V(G)$ with $|B|\geq k$, and $\mathcal{P}$  is map assigning  to each $2$-subset  $\{x,y\}$ of $B$ a set of at least $l$ pairwise internally disjoint paths in $G$ from $x$ to $y$. We write $\mathcal{P}_{\{x,y\}}=\mathcal{P}(\{x,y\})$.
We denote by $V(\mathcal{P}_{\{x,y\}})$ the union of the interiors of the paths  that are elements of $\mathcal{P}_{\{x,y\}}$.
We say that $(B,\mathcal{P})$ is {\em strong} if for all distinct $2$-subsets $\{x,y\}, \{x',y'\}$ of $B$, we have $V(\mathcal{P})\cap V(\mathcal{P})=\emptyset$; that is, each path $P\in \mathcal{P}_{\{x,y\}}$ is  internally disjoint from each path $P'\in \mathcal{P}_{\{x',y'\}}$.

We need the following result that was essentially proved in \cite{TW17}.
\begin{lemma}\label{lem:comp_model_rigid}
  For all $s,\rho,\sigma \in \nat$ there exist positive integers 
  $f=f(s,\rho,\sigma)$ and $g=g(s,\rho,\sigma)$ with the following property. Let $G$ be a graph and let $(B,\mathcal{Q})$ be a strong
  $(f,g)$-block in $G$ such that  $B$ is a stable set and for every $\{x,y\}\subseteq B$, the paths $(Q^*: Q\in \mathcal{Q}_{\{x,y\}})$ are pairwise anticomplete in $G$. Then one of the following holds.
\begin{enumerate}[(a)]
 \item There is an induced subgraph of $G$ isomorphic to a proper subdivision of $K_s$.
 \item There is a linear induced
   $K_{\rho,\sigma}$-model in $G$. 
\end{enumerate}
\end{lemma}

The difference between \cref{lem:comp_model_rigid} here and Lemma 3.6 of \cite{TW17} is that in \cite{TW17} it is assumed that $G$ is $K_{t+1}$-free (and $t$ is another parameter in the statement of the theorem), but there is no assumption that the set $B$ is stable. However, the only place in the proof where the bound on the clique number of $G$ is used is an application of 
Ramsey Theorem to conclude that a large subset of $B$ is stable, so we do not need that assumption here.

We also need the following, which follows immediately from the main result of  result of \cite{TW16}. Following \cite{TW16}, we say that a
{\em $(s,l)$-constellation} is a graph $C$ in which there is a stable set $S_{C}$ with $|S_C|=s$, such that $C \setminus S_C$ has exactly $l$ components, 
every component of $C\setminus S_{C}$ is a path, and every vertex
of $S_{C}$ has at least one neighbor in each component of
$S \setminus S_{C}$.
\begin{theorem}\label{thm:TW16main}
  For all  $s,l,r \in \nat$, there is a  positive integer
  $a=a(s,l,r)$ with the following property. Let $G$ be a graph that  that contains a $K_{a, a}$-induced-minor. Then one of the following holds.
   \begin{enumerate}[\rm (a)]
        \item\label{thm:motherKtt_a}There is an induced subgraph of $G$ isomorphic to either $K_{r,r}$, a subdivision of $W_{r\times r}$, or the line graph of a subdivision of $W_{r\times r}$.
        \item\label{thm:motherKtt_b} There is an
          $(s,l)$-constellation in $G$. 
        \end{enumerate}
  \end{theorem}

We are now ready to prove \cref{thm:slim}.
\begin{proof}[Proof of \cref{thm:slim}.]
  Let $G \in \mathcal{C}_t$.
  Let  $\rho=a(t,t,t)$  from
  \cref{thm:TW16main}. Let 
  $q=g(t^2,\rho,\rho)+f(t^2,\rho, \rho)$ and
  $p=f(t^2,\rho,\rho)$ as in \cref{lem:comp_model_rigid}.
  We will show that $G$ is $(p,q)$-slim.
  Suppose that there is a stable set $B\subseteq V(G)$ with
  $|B|=p$ such that every pair $x,y$ in $B$ is $q$-wide.
  
  Since $|B| \leq p$, it follows that for every $x,y \in B$ there exists
  a set $\mathcal{Q}_{\{x,y\}}$ of $x \dd y$-paths whose interiors are disjoint from $B$ and pairwise anticomplete with $|\mathcal{Q}_{\{x,y\}}| =g(t^2,\rho,
  \rho)$. Let $G'$ be the graph obtained from $B \cup \bigcup_{x,y \in B}V(\mathcal{Q}_{\{x,y\}})$
  by replacing each vertex $v$ in  $\bigcup_{x,y \in B}V(\mathcal{Q}_{\{x,y\}})$
  with a clique consisting of vertices $v_{\{x,y\}}$ for every
  $x,y \in B$ such that $v$ belongs to a path of $\mathcal{Q}_{\{x,y\}}$.
  Now for every $x,y \in B$, let $\mathcal{Q}'_{\{x,y\}}$ be the set of paths
  obtained by replacing each vertex $v$ in each path of $\mathcal{Q}_{\{x,y\}}$
  by its copy $v_{\{x,y\}}$. Let $\mathcal{Q}'=\bigcup_{x,y \in B}\mathcal{Q}'_{\{x,y\}}$. Then $(B,\mathcal{Q'})$ is an $(f(s,\rho, \rho), g(s,\rho,\rho))$
  strong block satisfying the assumptions of \cref{lem:comp_model_rigid}.
  It follows from \cref{lem:comp_model_rigid} that one of the following holds.
  \begin{enumerate}[(a)]
 \item\label{lem:comp_model_rigid_a}  There is an induced subgraph of $G'$ isomorphic to a proper subdivision of $K_t$.
 \item \label{lem:comp_model_rigid_b} There is a linear induced
   $K_{\rho,\rho}$-model in $G'$. 
\end{enumerate}
  Suppose first that $G'$ contains an induced subgraph $F$ isomorphic to a proper subdivision of $K_t$.
  %\todo[inline]{I'm not convinced of the following sentence as its currently written... i.e. if a vertex of $G$ gets replaced by a $K_2$, then one of the paths of $F$ may not be as long in $G$ as it is in $G'$ (but I still think it is true that there is \textit{some} proper subdivision of $K_t$ in $G$, it just might not be $F$)}
  Then no two vertices of $F$ are adjacent twins, and so
  $F$ is isomorphic to an induced subgraph of $G$. It follows that $G$ contains a subdivision of $W_{t,t}$, contrary to the fact that $G \in \mathcal{C}_t$.
  Thus we deduce that that there is a linear induced  $K_{\rho,\rho}$-model 
  $(C_1, \dots, C_{2\rho})$  in $G'$.
  
  %Then $\bigcup_{i=1}^{\rho+\sigma}C_i$ is $K_4$-free. ********** 

  We apply \cref{thm:TW16main} to $G'[\bigcup_{i=1}^{2\rho}C_i]$
    to deduce that $G'$ contains 
  an induced subgraph $F$ isomorphic to either $K_{t,t}$, a subdivision of $W_{t\times t}$, or the line graph of a subdivision of $W_{t\times t}$
  or a $(t,t)$-constellation. In all cases, 
  no two vertices of $F$ are adjacent twins, and so $F$ is isomorphic to an induced subgraph of $G$. But in all cases $F$ contains either an induced
  $W_{t \times t}$-model, or an induced $K_{t,t}$-model, 
  contrary to the fact that  $G \in \mathcal{C}_t$.
  \end{proof}

\section{Separating slim pairs of vertices in $(F,r)$-based graphs in $\mathcal{C}_t$}\label{sec:based}

%\todo[inline]{I think the notation with $F$ is a bit broken in this paragraph: $S \in F$ is used for stars $S$, but since $F$ is an induced subgraph then $S \in F$ should technically be a vertex since we use vertex sets and ISGs interchangeably}
Let $G$ be a graph and let $F$ be an induced subgraph of $G$ such that every component of $F$ is a star (with at least two vertices).
For every star $S$  of $F$, let {\em the center} of $S$, denoted by $c(F)$, be defined as follows. If $|S|=2$, then let $c(S)$ be an arbitrary vertex of $S$.
If $|S|>2$, let $c(S)$ be the vertex of degree greater than one in $S$.
Let $l(S)=S \setminus c(S)$.
Let $C(F)$ be the set of all centers of stars of $F$; we call the vertices of
$C(F)$ the {\em centers of $F$}. Let $L(F)=F \setminus  C(F)$; we call
the vertices of $L(F)$ the {\em leaves} of $F$. Observe that both
$C(F)$ and $L(F)$ are stable sets. For a set $X \subset V(G)$, we say that
$X$ is {\em $F$-based} if $S \subset X$ for every star  $S$ of $F$ with
$S \cap X \neq \emptyset$. We define the {\em $F$-measure} of an
$F$-based subset $X$, $\mu_F(X)$, as follows. Let
$a(X)=|\{ S \text{ such that } S \text{ is a star of } F \text { and }S \cap X \neq \emptyset\}|$.
  Let $b(X)=|X \setminus F|$. Then $\mu_F(X)=a(X)+b(X)$.
%\todo[inline]{here the notation $S\in F$ seems to implies that $F$ is a set of stars and not a subgraph}
  Let $G'$ be an induced subgraph of $G$ and let $S$ be a star of $F$. If
  $c(S) \in G'$ and $l(S) \cap G' \neq \emptyset$, the
  {\em projection of $S$ onto $G'$}, $p_{G'}(S)$,  is the star $S \cap G'$.
  We define  $F(G')$ to be the induced subgraph of $G'$  whose components are $p_{G'}(S)$ where  $S$ is a star of $F$ with
  $c(S) \in G'$ and $l(S) \cap G' \neq \emptyset$.
  We define $C(F(G'))=C(F) \cap F(G')$ and $L(F(G'))=L(F) \cap F(G')$.
  Let $r$ be a positive integer. 
We say that $G$ is {\em $(F,r)$-based} if every induced subgraph $G'$ of $G$
%either has $tw(G') \leq r$, or 
admits a tree decomposition $(T', \chi')$ such that for each $t \in V(T')$,
$\chi'(t')$ is an $F(G')$-based set with $\mu_{F(G')}(\chi'(t')) \leq r$.
We call $(T',\chi')$ an  {\em $(F,r)$-tree decomposition} of $G'$.
The goal of this section is to prove that, for appropriately chosen $q$,  every $q$-slim pair of vertices in an $(F,r)$-based graph in $\mathcal{C}_t$ has a small separator.

We start with the following:
\begin{theorem}
\label{thm:mineslim}
  Let $G$ be a graph, let $F$ be an induced subgraph of $G$ such that every component of $F$ is a star (with at least two vertices), and let $r,q,x \in \nat$.
  Assume that $G$ is $(F,r)$-based.
  Let $a,b \in V(G)$ be an $q$-slim pair in $G$. Then there exists a set 
  $D \subseteq V(G)$ with $|D| \leq r(q-1)x$ such that either
  \begin{enumerate}
  \item $D$ is an $a \dd b$-separator in $G$, or
  \item  $(a,b)$ is $(x, r(q-1), 1,1)$-mineable in $G \setminus D$.
  \end{enumerate}
\end{theorem}

We start with a lemma.
\begin{lemma}
  \label{lem:basket}
  Let $G$ be a graph, let $q \in \nat$, and let $a,b \in V(G)$ be a $q$-slim pair.
  Let $(T, \chi)$ be a
  tree decomposition of $G$. Then there exists $I \subseteq V(T)$
with $|I|<q$ such that 
    for every $a \dd b$-path $P$ in $G$ we have that
  $(\bigcup_{t \in T}\chi(t))  \cap P^* \neq \emptyset.$ 
  \end{lemma}

For $q=3$ this  is Theorem 2.6 of \cite{tw15}, but the same proof works for all
$q$.

We can now prove \cref{thm:mineslim}.

\begin{proof}[Proof of \cref{thm:mineslim}]
  We may assume that $a$ and $b$ are in the same component of $G$, for otherwise the empty set is in $a \dd b$-separator in $G$. Let $G_0=G$. 
  %If $tw(G_0) \leq r$, then by \cref{lem:basket}
  %there is an $a \dd b$-separator in $G$ of size at most $(q-1)r$, thus we may assume %that $tw(G_0)>r$. 
  Let $(T_0,\chi_0)$ be an $(F,r)$-tree decomposition of $G_0$.
Let $I_0$ be as in \cref{lem:basket}. Let\\
$C_1=\bigcup_{t \in I_0}(\chi(t) \setminus F)$,\\ 
$Y_1=\bigcup_{t \in I_0}(\chi(t) \cap C(F))$ and\\
$X_1=\bigcup_{t \in I_0}(\chi(t) \cap L(F))$ .\\
Then $Y_1 \subseteq C(F)$,
$X_1 \subseteq N(Y_1) \cap L(F)$ and
$Y_1 \cup X_1$ is
an $(a,b)$-separator in $G_0 \setminus C_1$. 
Since $Y_1 \subseteq C(F)$, it follows that $Y_1$ is a stable set, and so every component of $Y_1$ has size one.
Moreover $|C_1 \cup Y_1| \leq r(q-1)$. Let $G_1=G_0 \setminus (C_1 \cup Y_1)$. We may assume that $a$ and $b$ are in  the same component of $G_1$, for otherwise the theorem holds setting $D=C_1 \cup Y_1$.

We proceed as follows.  Assume that for some $i \geq 1$ we have defined $G_i,C_i, Y_1, \dots, Y_i, X_1, \dots, X_i$ with the following properties:
\begin{enumerate}
\item $G_i=G \setminus (C_i \cup \bigcup_{j \leq i} Y_j)$.
\item $a$ and $b$ are in the same component of $G_i$.
\item $|C_i \cup \bigcup_{j \leq i} \cup Y_j)| \leq  i(q-1)r$.
  \item $|Y_j| \leq r(q-1)$ for all $j \leq i$.
\item $Y_1, \dots, Y_i$  are disjoint subsets of $C(F)$.
\item $\bigcup_{j=1}^iY_i$ is a stable set.
\item For every $j<i$, $X_j \subseteq L(F) \cap N(Y_j)$.
  \item $X_1, \dots, X_i$ are disjoint subsets of $L(F)$.
    \item $Y_i \cup X_i$ is an $a \dd b$-separator in
    $G \setminus (C_i \cup \bigcup_{j<i}Y_j)$.
\end{enumerate}
Note that $G_1,C_1,X_1,Y_1$ satisfy the conditions above.
If $i=x$, we stop.
If $i<x$, proceed as follows.
%Suppose that  $tw(G_i) \leq r$, then by \cref{lem:basket}
 % there is an $a \dd b$-separator $C$ in $G_i$ of size at most $(q-1)r$. 
%Now $C_i \cup \bigcup{j \leq i}Y_i \cup C$ is an $a \dd b$-separator in $G$
%of size at most $(i+1)(q-1)r \leq x(q-1)r$. Thus we may assume that
%$tw(G_i)>r$.

Let $(T_i,\chi_i)$ be an $(F,r)$-tree decomposition of $G_i$.
Let $I_i$ be as in \cref{lem:basket}. Let\\
$C_{i+1}=C_i \cup \bigcup_{t \in I_i}(\chi(t) \setminus F(G_i))$,\\ 
$Y_{i+1}=\bigcup_{t \in I_i}(\chi(t) \cap C(F(G_i)))$ and\\
$X_{i+1}=\bigcup_{t \in I_i}(\chi(t) \cap L(F(G_i)))$.\\
Then $X_{i+1} \subseteq N(Y_{i+1}) \cap L(F)$ and 
$Y_{i+1} \cup X_{i+1}$ is
an $(a,b)$-separator in $G \setminus (C_{i+1} \cup \bigcup_{j < i+1}Y_j)|  =G_i \setminus C_{i+1}$.
Moreover, $|Y_{i+1}| \leq r(q-1)$ and
$$|C_{i+1} \cup \bigcup_{j \leq i+1}Y_j| \leq |C_i \cup \bigcup_{j \leq i}Y_j| + |\bigcup_{t \in I_i} \chi(t) \setminus L(F(G_i)|\\
\leq (i-1)r(q-1)+r(q-1)= i(q-1)r.$$
It follows from the definitions of $C(F(G_i))$ and $L(F(G_i))$  that
$Y_{i+1}$ is disjoint from $\bigcup_{j \leq i}Y_j$, and $X_{i+1}$ is
disjoint from $\bigcup_{j \leq i}X_j$.
Let $G_{i+1}=G_i \setminus (C_{i+1} \cup \bigcup_{j \leq i+1}Y_j)$.
If $a$ and $b$ are in different components of $G_{i+1}$, then
$C_{i+1} \cup \bigcup_{j \leq i+1}Y_j$ is an $a \dd b$-separator of size
at most $(i+1)(q-1)r \leq x(q-1)r$ in $G$, and the theorem holds.
Thus we may assume that  $a$ and $b$ are in the same component of $G_{i+1}$.
Now $G_{i+1}, C_{i+1}, Y_1, \dots Y_{i+1}, X_1, \dots, X_{i+1}$ satisfy the conditions above, and we can continue the process.

It follows that  we may assume that $i=x$.  
Now, setting $D=C_i$, 
the sets $Y_1, \dots, Y_x, X_1, \dots, X_x$ show that $(a,b)$ is $(x, (q-1)r, 1,1)$-mineable in $G \setminus D$, as required.
\end{proof}

We can now prove the main result of this section. Let $\phi$ be as in \cref{thm:mineabletobarrier}. We define $\psi(t,q) =\frac{100^2}{99}(4+2t)^2 \phi(t)(q-1)$.
\begin{theorem} \label{thm: sep slim general version}
  Let $t \in \nat$ and let $p,q$ be as in \cref{thm:slim}. Let $r\in \nat$ be such that $r\geq \psi(t,q)$.
   Then, for every $(F,r)$-based graph $G \in \mathcal{C}_t$ on $n$ vertices, and every  $q$-slim pair $(a,b)$ in $G$, there exists an $a \dd b$ separator in $G$ of size at most $2^{5(\log(r)+\sqrt{\log n\log r})}$.
\end{theorem}

%\todo[inline]{(Julien) I've splitted this theorem into $2$}
\begin{proof}
  By \cref{thm:slim} $G$ is $(p,q)$-slim.
  Let $\phi  : \mathbb{N} \rightarrow \mathbb{N}$ be as in \cref{thm:mineabletobarrier}.
  Let $x=2^{\sqrt{\log(n)\log(r)}} 100 (4+2t)^2 \leq 2^{\sqrt{\log(n)\log(r)}}  r$. 
  We define the following functions.
  $$c(n)= \frac{100}{99} (4+2t) \phi(t)\,r(q-1)\, t \ \leq \frac{r^2}{100}.$$
  $$f(n)=2xr(q-1) \leq \frac{xr^2}{100}.$$
  $$g(n)=\frac{100n (4+2t)^2}{x} \leq 2^{-\sqrt{\log(n)\log(r)}}.$$
  \\
  \\
  \sta{$G$ is $(q,t, 1, c,f,g)$-slim-barred.\label{slimtoslimbarred}}

  Let $(c,d)$ be a $q$-slim pair in $G$. 
  Since $G$ is $(F,r)$-based, \cref{thm:mineslim} implies that
   there is $D \subseteq V(G)$ with $|D| \leq xr(q-1)$ such that  either
\begin{enumerate}
  \item $D$ is a $c \dd d$-separator in $G$, or
  \item  $(c,d)$ is $(x, r(q-1), 1,1)$-mineable in $G \setminus D$.
  \end{enumerate}
  If $D$ is a $c \dd d$-separator in $G$, then
  setting $X=Y=Z=C=\emptyset$ and $M = D$ gives a $(t,1)$ barrier in $G \sm M$. Thus, we may assume that
  $(c,d)$ is $(x, (q-1)r, 1,1)$-mineable in $G \setminus D$.
  Since $(c,d)$ is $q$-slim in $G$ and $G \in \mathcal{C}_t$,
 \cref{thm:mineabletobarrier} implies that  there
 exist disjoint subsets $X,Y,Z,C,M'$ of $V(G) \setminus (D \cup  \{c,d\})$ such that
    \begin{enumerate}
    \item $B=(X,Y,Z,C)$ is a $t$-barrier in $G\setminus M'$ that separates $c$ from $d$.
    \item $|C|\leq \frac{100}{99} (4+2t) \phi(t)\,r(q-1)\, t$.
    \item $|M'|\leq xr(q-1)$.
    \item $|X\cup Y\cup Z| \leq \frac{100n (4+2t)^2}{x}$.
\end{enumerate}
Setting $M=M' \cup D$, we get that   there
 exist disjoint subsets $X,Y,Z,C,M$ of $V(G) \setminus \{c,d\}$ such that
    \begin{enumerate}
    \item $B=(X,Y,Z,C)$ is a $t$-barrier in $G\setminus M$ that separates $c$ from $d$.
    \item $|C|\leq \frac{100}{99} (4+2t) \phi(t)\,r(q-1)\, t$.
    \item $|M|\leq 2xr(q-1)$.
    \item $|X\cup Y\cup Z| \leq \frac{100n (4+2t)^2}{x}$.
\end{enumerate}
      This proves
  \eqref{slimtoslimbarred}.
  \\
  \\
  Now since $(a,b)$ is $q$-slim and $G$ is $(q,t,c,f,g)$-slim-barred, \cref{thm:slimbarredtosep} implies
  that there is an $a \dd b$-separator of size at most $20(f(n)+3c(n)^2) c(n)^{2 \frac{\log n}{\log{\frac{n}{g(n)}}}}$  in $G$. 
    We now bound this quantity.
    \begin{align*}
    20(f(n)+3c(n)^2) c(n)^{2 \frac{\log n}{\log{\frac{n}{g(n)}}}}&= 20(f(n)+3c(n)^2) c(n)^{2 \sqrt{\frac{\log n}{\log r}}}\\
    &\leq 20\left(\frac{xr^2}{100}+\frac{3r^4}{100}\right) c(n)^{2 \sqrt{\frac{\log n}{\log r}}} \\
    &\leq xr^4 c(n)^{2 \sqrt{\frac{\log n}{\log r}}}\\
    &\leq xr^{4} 2^{4\log(r)\sqrt{\frac{\log n}{\log r}}}\\
    &\leq xr^{4} 2^{4\sqrt{\log n\log r}}\\
    &\leq r^{5} 2^{5\sqrt{\log n\log r}}\\
    &=  2^{5(\log(r)+\sqrt{\log n\log r})}.
    \end{align*}
  \end{proof}

\section{Bounding the treewidth of $(F,r)$-based graphs in $\mathcal{C}_t$}\label{sec:twbased}

In this section, we complete the treatment of $(F,r)$-based graphs, proving the following. Recall that $\psi(t,q) =\frac{100^2}{99}(4+2t)^2 \phi(t)(q-1)$,
where $\phi$ is as in \cref{thm:mineabletobarrier}.

\begin{theorem} \label{thm:specialtwbound}
  Let $t \in \nat$, let $p,q$ be as in \cref{thm:slim}, and let $r\geq \psi(t,q)$.
   Then every $n$-vertex 
  $(F,r)$-based graph $G$ in $\mathcal{C}_t$ satisfies $tw(G) \leq2^{9\log(r)+5\sqrt{\log n\log r}}$.
\end{theorem}

  We remark that a version of \cref{thm:specialtwbound}  
 could be proved using  
  \cref{thm: sep slim general version}  and Theorem 6.5 of
  \cite{tw7} if the clique number of $G$ is bounded. Here we will include a
  different proof that does not use this assumption.

For a graph $G$ a function  $w: V(G) \rightarrow [0,1]$ is a {\em normal weight function} on $G$ if $w(V(G))=1$, where for $X \subseteq V(G)$ we denote $\sum_{v \in X} w(v)$ by $w(X)$.
Let $c \in [0, 1]$ and let $w$ be a normal weight function on $G$. A set $X \subseteq V(G)$ is a {\em $(w,c)$-balanced separator} if
$w(D) \leq  c$ for every component $D$ of $G \setminus X$. The set $X$ is a {\em $w$-balanced separator} if $X$ is a $(w,\frac{1}{2})$-balanced separator.

The following result was originally proved by Robertson and Seymour in \cite{RS-GMII},
  and tightened by Harvey and Wood in \cite{params-tied-to-tw}.
  It was then  restated and proved in the language of $(w, c)$-balanced separators in \cite{tw2}. 

  \begin{theorem}\label{thm:septotw}
      Let $G$ be a graph, let $c \in [\frac{1}{2}, 1)$, and let $d$ be a positive integer. If for every
    normal weight function $w: V(G) \to [0, 1]$,  $G$ has a $(w, c)$-balanced separator of size at most $d$, then $tw(G) \leq \frac{1}{1-c}d$. 
\end{theorem}

  We now prove the main result of this section.
  \begin{proof}[Proof of \cref{thm:specialtwbound}]
    We start with the following.

    \sta{Let $w: V(G) \rightarrow [0,1]$ be a normal weight function on $G$. Then
      $G$ admits a $(w,\frac{1}{2})$-balanced separator of size at most
      $ rp+(rp)^2 2^{5(\log(r)+\sqrt{\log n\log r})}$ . \label{smallsep}}

    Let $G_1=G$, $X_0=\emptyset$ and  $w_1=w$. For $i \in \{1, \dots, p\}$,
    Assume that $X_1, \dots, X_{i-1}, G_1, \dots, G_i$ have already been defined and 
    proceed as follows. 
        Let $(T_i, \chi_i)$ be an $(F,r)$-tree decomposition of $G_i$.  It is well known (see e.g. the proof of 
        Lemma 7.19 in \cite{cygan2015parameterized} or Theorem 2.7 in \cite{tw15}) that there exists
        $t_i \in V(T_i)$ such that $w_i(D) \leq \frac{1}{2}$ for every component $D$ of $G \setminus \chi_i(t_i)$. Let $X_i=\chi_i(t_1) \setminus L(F(G_i))$.
        Then
    $|X_i| \leq r$. 

        If $i<p$, define
        $G_{i+1}=G_i \setminus X_i$
        and $w_{i+1}(v)=\frac{w_i(v)}{1-\Sigma_{u \in X_i}w_i(u)}$
        for every $v \in G_{i+1}$. Then $w_{i+1}$ is a normal weight function on
        $G_{i+1}$, and we repeat the process to define $X_{i+1}$.
        We stop when $i=p$ and $X_1, \dots, X_p$ have been defined.

Next, let $\mathcal{P}$ be the set of all $q$-slim pairs $(v,v')$  such that $v \in X_i$ and $v' \in X_{i'}$ for some $1 \leq i < i'  \leq p$.
For every  $(v,v') \in \mathcal{P}$, let $X_{vv'}$  
be a $v \dd v'$-separator in $G$
of size at most $2^{5(\log(r)+\sqrt{\log n\log r})}$ given by \cref{thm: sep slim general version}.
Let $X=\bigcup_{i=1}^pX_i \cup \bigcup_{(v,v') \in \mathcal{P}}X_{vv'}$. 
Then $|X|  \leq rp+(rp)^2 2^{5(\log(r)+\sqrt{\log n\log r})}$.
We show that $X$ is a balanced separator in $G$.
Suppose for contradiction that there is a component $D$ of 
                $G \setminus X$ with $w(D)>\frac{1}{2}$. Then $w_i(D) > \frac{1}{2}$ 
        for every $i$. Since $\chi(t_i)$ is a $w_i$-balanced separator in
        $G_i$ for every $i$, it follows that for every $i \in \{1, \dots, p\}$
        there is $v_i \in X_i \cap C(F(G_i))$ such that $v_i$ has
        a neighbour in $D$ (in fact, $v_i$ has a neighbour in
        $D \cap \chi_i(t_i) \cap L(F(G_i))$).  By \cref{thm:slim} there exist
        $v,v' \in \{v_1, \dots, v_p\}$ such that $(v,v') \in \mathcal{P}$.
        But there is a path from $v$ to $v'$ with interior in $D$, contrary to the fact that  $D \subseteq G \setminus X_{vv'}$.
        This proves \eqref{smallsep}.
        \\
        \\
        From  \eqref{smallsep} and \cref{thm:septotw} we deduce that
        \begin{align*}
        tw(G) &\leq 2\left(rp+(rp)^2 2^{5(\log(r)+\sqrt{\log n\log r})}\right)\\
        &\leq 4(rp)^2 2^{5(\log(r)+\sqrt{\log n\log r})} \\
        &\leq r^4 2^{5(\log(r)+\sqrt{\log n\log r})} \\
        &= 2^{9\log(r)+5\sqrt{\log n\log r}}.
        \end{align*}
  \end{proof}

\section{Bounding the treewidth of $\mathcal{C}_t^*$}\label{sec:useclique}

In this section, we prove the main result of this paper, which we restate.

\main*

 We follow the general road map of \cite{BHKM}, but we  phrase our arguments in a slightly different language. A {\em star-coloring} of
 a graph $G$ is a proper coloring such that the union of every two color classes induces a star forest in $G$ (that is a graph where each component is a star or a singleton). The {\em star chromatic number} of $G$ is the minimum $k$ such that $G$ admits a star coloring with $k$ color classes.  The following is immediate from Theorem~11 of \cite{BHKM} and Ramsey Theorem:

 \begin{theorem}
   \label{thm:starcoloring}
   For every $t \in \nat$ there is an integer $d_1=d_1(t)$ such that every
   graph in $C_t^*$ has star-chromatic number $d_1$.
 \end{theorem}

 We will also use the main result of \cite{BHKM}:

 \begin{theorem}
   \label{thm:degree}
   There exists an integer $d_2$ such that 
   for all $t, \Delta \in \nat$, every graph in
   $\mathcal{C}_t^*$ with maximum degree at most $\Delta$ has
   treewidth at most $(t\Delta)^{d_2}$.
   \end{theorem} 
   
 Finally, we  need the following result of \cite{ChuzhoyDeg3}:
 \begin{theorem}
   \label{thm:subcubic}
   There is an integer $d_3$ such that every graph $G$ contains a subcubic subgraph of treewidth at least $\frac {tw(G)}{\log^{d_3}(tw(G))}$.
 \end{theorem}

 \begin{proof} [Proof of \cref{thm:twbound}]
   Let $d_1$ be as in \cref{thm:starcoloring}, and let $C_1, \dots, C_{d_1}$ be
   the color classes of a star coloring of $G$. For $i,j \in \{1, \dots, d_1\}$
   let $F'_{ij}=G[C_i \cup C_j]$. Then  $F'_{ij}$ is a star forest. Let 
   $F_{ij}$ be the graph obtained from $F'_{ij}$ be removing all isolated vertices; write $E_{ij}=E(F_{ij})$. Note that $\bigcup_{i,j \in \set{1,\ldots,d_1}} E_{ij} = E(G)$; we call $\{E_{ij}\}_{i,j \in \{1, \dots,d_1\}}$
     {\em the corresponding edge partition}.
   We say that $E_{ij}$ is {\em active} in $G$
   if some vertex of $G$ is incident with more than three edges of $E_{ij}$.
   We define the {\em dimension} of $G$ to be the number of pairs $(i,j)$ for
   which $E_{ij}$ is active, and       denote the dimension  of $G$ by $dim(G)$.
    Let $p,q$ be as in \cref{thm:slim}, let
       $d_2$ be as in \cref{thm:degree}, and let $r_0=\max\set{\psi(t,q), (3td_1^2)^{d_2}}$.
    
      \sta{If $dim(G)=0$, then $tw(G) \leq r_0$.\label{dim0}}
      Since $dim(G)=0$, it follows that the maximum degree of $G$ is at most $3d_1^2$. Now \eqref{dim0} follows from \cref{thm:degree}. 
      \\
      \\
      Let $d_3$ be as in \cref{thm:subcubic}.
      Next, we recursively define a sequence of integers $r_1, \dots, r_{{d_1}^2}$.
      Having defined $r_1, \dots, r_i$, let 
$r_{i+1}=2^{9\log(r_i\log^{d_3}n) + 5\sqrt{\log(r_i\log^{d_3}(n))\log(n))}}$.
      
     % $r_{i+1}$ be the maximum of
     % \begin{itemize}
     % \item $r_i\log^{d_3}n$, and
%      \red{
 %     \item  $2^{\log^{1-\delta} n}$, where $\delta$ is as in
 %       \cref{thm:specialtwbound} applied with $r=r_i\log^{d_3}n$ and $t$ unchanged.}
      %  \item \blue{$2^{9\log(r_i\log^{d_3}n) + 5\sqrt{\log(r_i\log^{d_3}(n))\log(n))}}$}
      %  \end{itemize}

We will prove by induction on $dim(G)$ that $tw(G) \leq r_{dim(G)}$.
The base case is \eqref{dim0}, thus we assume that we have proved the result for  graphs of dimension at most $i$, and  that $dim(G)=i+1$.
Let $i_0,j_0 \in {1, \dots d_1}$ be such that $E_{i_0j_0}$ is active in $G$.

\sta{Let $G_0$ be an induced subgraph of $G$  with $tw(G_0)> r_i \log^{d_3}n$.
Let $G'$ be the graph obtained from $G_0$ by contracting the edges of
$F_{i_0j_0}(G_0)$. Then $tw(G') \leq r_i \log^{d_3}n$. \label{largetw}}

Suppose not. By \cref{thm:subcubic} $G'$ contains a subcubic subgraph $G''$
with $tw(G'') >  r_i$. The third and fourth paragraphs of the proof of Lemma~7
of \cite{BHKM} show how to construct an induced subgraph $H$ of $G_0$ that contains $G''$ as a minor, and such that every vertex of $H$ is incident with at
most three edges of $E_{i_0j_0}$. It follows that $tw(H) \geq tw(G'')>r_i$.

We now show that $dim(H) \leq i$ and get  contradiction.
First observe that $C_1 \cap V(H), \dots, C_{d_1^2}\cap V(H)$ is a star coloring of $H$ with corresponding edge partition
$\{E_{pq} \cap E(H))\}_{p,q \in \{1, \dots, d_1\}}$.
It follows that the set $E_{i_0j_0} \cap V(H)$ is not active in $H$.
Moreover, if the set $E_{pq}$ is not active in $G$, then the set $E_{pq} \cap V(H)$ is not active in $H$. Since $E_{i_0j_0}$ is active in $G$, we deduce that $dim(H)<dim(G)=i+1$,
a contradiction. This proves~\eqref{largetw}.
\\
\\
\sta{$G$ is $(F_{i_0j_0}, r_i\log^{d_3}n)$-based. \label{based}}

Let $G_0$ be an induced subgraph of $G$. We need to show that $G_0$
admits an $(F_{i_0j_0}, r_i\log^{d_3}n)$-based tree decomposition.
Let $G'$ be the graph obtained from $G_0$ by contracting the edges of
$F_{i_0j_0}(G_0)$. 
Then every vertex of $G'$ is either a vertex of $G$ or corresponds to a
component of $F_{i_0j_0}(G_0)$; we call the latter kind of vertex a
{\em star vertex}.
Let $(T,\chi')$ be a tree decomposition of $G'$ of width $tw(G')$.
We construct a tree decompositon $(T, \chi_0)$ of $G_0$, where
for every $t \in T$, $\chi_0(t)$ is obtained from $\chi'(t)$
by replacing every star
vertex with  the component of $G(F_{i_0j_0})$ to which it corresponds.
Since we know by \eqref{largetw} that $tw(G') \leq r_i \log^{d_3}n$, it follows  that $(T_0, \chi_0)$ is an $(F_{i_0j_0}, r_i\log^{d_3}n)$-based tree decomposition of $G_0$, and \eqref{based} follows.
\\
\\
Now, in view of \eqref{based},  \cref{thm:specialtwbound} implies that $tw(G) \leq r_{i+1}$, as required. This completes the inductive proof that $tw(G) \leq r_{dim(G)}$.

We now bound $r_i$.
Let $c_0 = r_0 + 18d_3$.
Now we show that $r_i \leq 2^{16^i{c_0}\log^{1-1/2^i}n}$.
It is enough to prove that  $\log r_i \leq 16^i{c_0}\log^{1-1/2^i}n$.
To do so, we proceed by induction.
The base case trivially holds as $c_0\geq r_0$.
For $i \geq 1$, let $c_i=16^i c_0$ and $\epsilon_i = \frac{1}{2^i}$. 
By the induction hypothesis, we have that
%\todo[inline]{This would probably be easier to follow if we bounded $\log(r_{i+1})$ so that the right hand side could be just the exponents}
\begin{align*}
   \log  r_{i+1}&\leq 9\log(2^{c_i\log^{(1-\epsilon_i)}(n)}   \log^{d_3}n) + 5\sqrt{\left(c_i\log^{(1-\epsilon_i)}(n) + \log(\log^{d_3}(n))\right)\log(n))}\\
    &\leq 9c_i\log^{(1-\epsilon_i)}(n)  + 9 d_3\log\log n + 5\sqrt{c_i}\log^{(1-\epsilon_i/2)}(n) + 5\sqrt{d_3\log\log(n)\log(n))}\\
    &\leq 14 c_i \log^{1-\epsilon_i/2}(n) +  9 d_3\log\log n + 5\sqrt{d_3\log\log(n)\log(n)}\\
    &\leq 14 c_i \log^{1-\epsilon_i/2}(n) +  c_i \log^{1-\epsilon_i/2}(n)  +  c_i \log^{1-\epsilon_i/2}(n) \\
    &\leq 16 c_i \log^{1-\epsilon_i/2}(n) 
\end{align*}
Setting $c=c_{d_1^2}$ and $\epsilon = \epsilon_{d_1^2}$ completes the proof.

\end{proof}

\section{Acknowledgments}
We are grateful to Sepehr Hajebi and Sophie Spirkl for many inspiring discussions.

\bibliographystyle{abbrv}
\bibliography{ref}

%\printbibliography

\end{document}